\documentclass[a4paper,11pt,draft]{amsart}

\usepackage{amsmath}
\usepackage{amssymb}
\usepackage{ascmac}
\usepackage{amsthm}
\usepackage{color}
\usepackage{enumerate}
\usepackage{enumitem}
\setlist[enumerate,1]{label={\normalfont(\roman*)}}
\usepackage{mathrsfs}
\usepackage{graphicx}
\usepackage{braket}
\usepackage{wrapfig}
\usepackage{layout}
\usepackage[margin=1.9cm, bottom=3cm, footskip=1.5cm]{geometry}
\usepackage{layout}

\numberwithin{equation}{section}
\usepackage
[draft=false,setpagesize=false,pdfstartview=FitH,
colorlinks=true,linkcolor=blue,citecolor=magenta,pagebackref=true,bookmarks=false]
{hyperref}
\usepackage{cite}

\newcommand{\R}{\mathbf{R}}
\newcommand{\Z}{\mathbf{Z}}
\newcommand{\C}{\mathbf{C}}
\newcommand{\RN}{\mathbf{R}^N}
\newcommand{\N}{\mathbf{N}}
\newcommand{\e}{\varepsilon}
\newcommand{\rd}{\mathrm{d}}

\newcommand{\la}{\left\langle}
\newcommand{\ra}{\right\rangle}
\newcommand{\RE}{\mathrm{Re}\,}
\newcommand{\SN}{\mathbf{S}^{N-1}}
\def\ov#1{\overline{#1}}

\def\wt#1{\widetilde{#1}}
\newcommand{\ea}{e(\alpha)}
\newcommand{\ma}{\mathcal{M}(\alpha)}
\newcommand{\tea}{\wt{e}(\alpha)}
\newcommand{\tH}{\wt{H}}

\theoremstyle{definition}
\newtheorem{definition}{Definition}[section]
\theoremstyle{plain}
\newtheorem{theorem}[definition]{Theorem}

\newtheorem{lemma}[definition]{Lemma}
\newtheorem{proposition}[definition]{Proposition}
\theoremstyle{remark}
\newtheorem{remark}[definition]{Remark}

\usepackage[margin=1.9cm, bottom=3cm, footskip=1.5cm]{geometry}
\usepackage{layout}

\title[]{The compactness of minimizing sequences for a nonlinear Schr\"odinger system with potentials}
\author[]{Norihisa Ikoma}
\date{}
\address{Department of Mathematics, Faculty of Science and Technology, Keio University, 
	Yagami Campus: 3-14-1 Hiyoshi, Kohoku-ku, Yokohama 223-8522, Japan}
\email{ikoma@math.keio.ac.jp}

\author[]{Yasuhito Miyamoto}
\address{Graduate School of Mathematical Sciences, The University of Tokyo,
	3-8-1 Komaba, Meguro-ku, Tokyo 153-8914, Japan}
\email{miyamoto@ms.u-tokyo.ac.jp}


\begin{document}
	\begin{abstract}
		In this paper, we consider the following minimizing problem with two constraints:
			\[
				\inf \Set{ E(u) | u=(u_1,u_2), \ \| u_1 \|_{L^2}^2 = \alpha_1, \ \| u_2 \|_{L^2}^2 = \alpha_2 },
			\]
		where $\alpha_1,\alpha_2 > 0$ and $E(u)$ is defined by 
			\[
				E(u) := \int_{\mathbf{R}^N} 
				\left\{\frac{1}{2} \sum_{i=1}^2 
				\left( |\nabla u_1|^2 + V_i (x) |u_i|^2 \right) - \sum_{i=1}^2 \frac{\mu_i}{2p_i+2} |u_i|^{2p_i+2} 
				- \frac{\beta}{p_3+1} |u_1|^{p_3+1} |u_2|^{p_3+1} \right\} \mathrm{d} x.
			\]
		Here $N \geq 1$, $ \mu_1,\mu_2,\beta > 0$ and $V_i(x)$ $(i=1,2)$ are given functions. 
		For $V_i(x)$, we consider two cases: (i) both of $V_1$ and $V_2$ are bounded, (ii) one of 
		$V_1$ and $V_2$ is bounded. 
		Under some assumptions on $V_i$ and $p_j$, we discuss the compactness of any minimizing sequence. 
	\end{abstract}

\maketitle

\section{Introduction}
\label{section:1}

In this paper, we discuss the following minimizing problem: 
for $\alpha = (\alpha_1,\alpha_2) \in [0,\infty)^2$, 
	\begin{equation}\label{eq:1.1}
		\begin{aligned}
			\ea &:= \inf 
			\Set{ E(u) | u = (u_1,u_2) \in M(\alpha) }
			 (=e(\alpha_1,\alpha_2)) ,
			\\
			E(u) &:= \frac{1}{2} \sum_{i=1}^2 \int_{\RN} |\nabla u_i|^2 + V_i(x) |u_i|^2  \rd x 
			- \int_{\RN} \sum_{i=1}^2 \frac{\mu_i}{2p_i + 2} |u_i|^{2p_i+2} 
			+ \frac{\beta}{p_3+1} |u_1|^{p_3+1} |u_2|^{p_3+1} \rd x, 
			\\
			M(\alpha) &:= 
			\Set{ u = (u_1,u_2) \in H | 
				\| u_1 \|_{L^2}^2 = \alpha_1, \ \| u_2 \|_{L^2}^2 = \alpha_2}, \ 
			H := H^1 \times H^1, \ 
			H^1 := H^1(\RN,\C).
		\end{aligned}
	\end{equation}
Here $N \geq 1$, $0<p_1,p_2 ,p_3< 2/N$, $0< \mu_1,\mu_2,\beta$ are constants, 
$V_1(x), V_2(x) : \RN\to \R$ are given functions and $\| u_i \|_{L^2}^2 := \int_{\RN} |u_i|^2 \rd x$. 
Though we formally include the case $\alpha_1 = 0$ or $\alpha_2=0$ in \eqref{eq:1.1}, 
we are interested in the case $\alpha_1,\alpha_2>0$. 
The problem \eqref{eq:1.1} is related to a Schr\"odinger system 
	\begin{equation}\label{eq:1.2}
		\left\{\begin{aligned}
			&i \partial_t \psi_1 + \Delta \psi_1 - V_1(x) \psi_1 + \left( \mu_1 |\psi_1|^{2p_1} 
			+ \beta |\psi_1|^{p_3-1} |\psi_2|^{p_3+1}  \right) \psi_1 = 0 ,
			\\
			&i \partial_t \psi_2 + \Delta \psi_2 - V_2(x) \psi_2 + \left( \mu_2 |\psi_2|^{2p_2} 
			+ \beta |\psi_1|^{p_3+1} |\psi_2|^{p_3-1}  \right) \psi_2 = 0 
		\end{aligned}\right.
	\end{equation}
and \eqref{eq:1.2} appears in the field of the Bose--Einstein condensation and nonlinear optics 
\cite{BSSSC-04,CECSM-01,EGBB-97,HS-96,OS-98,PB-98}. 
If $u=(u_1,u_2)$ is a minimizer of \eqref{eq:1.1}, then 
there exist $\lambda_1,\lambda_2 \in \R$ such that 
	\begin{equation}\label{eq:1.3}
			\left\{\begin{aligned}
			&- \Delta u_1 + V_1 (x) u_1 - \left( \mu_1 |u_1|^{2p_1} + \beta |u_1|^{p_3-1} |u_2|^{p_3+1} \right) u_1 
			= - \lambda_1 u_1 & &\mathrm{in} \ \RN,
			\\
			&- \Delta u_2 + V_2 (x) u_2 - \left( \mu_2 |u_2|^{2p_2} + \beta |u_1|^{p_3+1} |u_2|^{p_3-1} \right) u_2 
			= - \lambda_2 u_2 & &\mathrm{in} \ \RN.
			\end{aligned}\right.
	\end{equation}
Hence, $u=(u_1,u_2)$ is a standing wave solution of \eqref{eq:1.2}, that is, a solution of \eqref{eq:1.2} of the form 
$(\psi_1(t,x), \psi_2(t,x)) = ( e^{i\lambda_1 t} u_1 (x) , e^{i \lambda_2 t} u_2(x) )$.

	Recently, the problem \eqref{eq:1.1} and related topics are studied in various settings. 
Let us consider the case $V_1(x) \equiv 0 \equiv V_2(x)$ and we refer to \cite{B-15-1,GJ-16,MMP-10,
NW-11,NW-13,Oh-96}. 
In particular, Gou and Jeanjean \cite{GJ-16} proved that any minimizing sequence for \eqref{eq:1.1} 
is compact in $H$ up to translations. 
On the other hand, Gou \cite{Go-18} treated \eqref{eq:1.1} with partial confinements when $N=3$, that is, 
$V_1(x_1,x_2,x_3) = V_2(x_1,x_2,x_3) = x_1^2 + x_2^2$, 
and showed the compactness of minimizing sequence up to translations in $x_3$. 
The problem \eqref{eq:1.1} with $V_1(x),V_2(x) \to \infty$ as $|x| \to \infty$ was also studied 
in \cite{GLWZ-19-1,GLWZ-19-2,NTV-15}, 
and in these papers, the properties of minimizers were investigated. 
For other minimizing problems with multi-constraint conditions, 
we refer to \cite{AlB-13,Ar-18,B-15-2,B-16,B-17,B-19,BCP-18,Ga-15,I-14,WY-18}.

	The aim of this paper is to treat bounded potentials, that is, 
we assume that at least one of $V_1$ and $V_2$ is bounded. 
More precisely, we consider two cases: (i) both $V_1$ and $V_2$ satisfy 
(V1') below, (ii) $V_1$ and $V_2$ satisfy (V2') below.
	\begin{enumerate}
		\item[(V1')] 
			$V \in C(\RN,\R) \cap L^\infty (\RN,\R)$, 
			$V_{\infty} := \lim_{|y| \to \infty} V(y)  $ and $V(x) \leq V_\infty$ for all $x \in \RN$. 
		\item[(V2')]
			$V_1$ satisfies (V1'), 
			$V_2 \in C(\RN,\R)$ and $V_2(x) \to \infty$ as $|x| \to \infty$. 
	\end{enumerate}
We note that under (V2'), $H$ and $\ea$ are replaced by 
	\[
		\begin{aligned}
			\tea
			&:= \inf_{ u \in \wt{M}(\alpha) } E(u), 
			\quad \wt{M} (\alpha) := \wt{H} \cap M(\alpha)
			,\\
			\tH_2 &:= \Set{ u_2 \in H^1(\RN) | \int_{\RN} V_2(x) |u_2|^2 \rd x < \infty }, \quad 
			\tH := H^1 \times \tH_2 .
		\end{aligned}
	\]
For exponents $p_1,p_2,p_3$, we introduce:
	\begin{enumerate}
		\item[(p1)] $0 < p_1,p_2,p_3 < 2/N$. 
		\item[(p2)] When $N \geq 5$, $0<p_1,p_2 \leq 1/(N-2)$. 
	\end{enumerate}

	We first consider the case where both of $V_1$ and $V_2$ are bounded:

\begin{theorem}\label{theorem:1.1}
	Assume that $V_1$ and $V_2$ satisfy \textup{(V1')}, either $V_1(x)$ or $V_2(x)$ is not a constant function, 
	\textup{(p1)}, \textup{(p2)} and $\alpha_i > 0$ $(i=1,2)$. 
	Then every minimizing sequence $(u_n)$ for $\ea$, namely, 
	$(u_n) \subset M(\alpha)$ and $E(u_n) \to \ea$, 
	has a strongly convergent subsequence in $H$. Hence, $\ea$ is attained 
	and the set of all minimizers is compact in $H$. 
\end{theorem}

	Next, let us consider the case where $V_1$ is bounded and $V_2(x) \to \infty$ as $|x| \to \infty$:

\begin{theorem}\label{theorem:1.2}
	Assume \textup{(V2')}, \textup{(p1)} and $\alpha_i > 0$ $(i=1,2)$. 
	Then, every minimizing sequence $(u_n)\subset \wt{M}(\alpha)$ for $\tea$ has 
	a strongly convergent subsequence in $\tH$.
	Hence, $\tea$ is attained and the set of all the minimizers is compact in $\tH$. 
\end{theorem}

	\begin{remark}
		\begin{enumerate}
			\item In Theorem \ref{theorem:1.2}, the condition (p2) is not necessary. 
			\item Even if we exchange the roles of $V_1(x)$ and $V_2(x)$, Theorem \ref{theorem:1.2} still holds. 
		\end{enumerate}
	\end{remark}

Finally we state one property of a set of minimizers in Theorems \ref{theorem:1.1} and \ref{theorem:1.2}. 
Since this property can be proved as in Montefusco, Pellacci and Squassina \cite{MPS-10}, we omit the proof for it. 
	\begin{proposition}\label{proposition:1.4}
	Let $\alpha_i > 0$ $(i=1,2)$ and set 
	\[
	\ma := \Set{ u \in M(\alpha) | 
	E(u) = \ea }.
	\]
	Then 
	\[
	\ma = \Set{ \left( e^{i \theta_1} w_1, e^{i \theta_2} w_2  \right) 
	 | 
	 0 \leq \theta_1,\theta_2 \leq 2 \pi, \ w_1 (x), w_2(x) > 0, \ (w_1,w_2) \in \ma }.
	\]
	The same claim holds for $\wt{e}(\alpha)$: 
	\[
		\begin{aligned}
			\wt{\mathcal{M}} (\alpha) 
			:&= \Set{  u \in \wt{M}(\alpha) | 
				E(u) = \tea  
			}
		\\
			&= \Set{ \left( e^{i \theta_1} w_1 , e^{i \theta_2} w_2  \right) | 
				0 \leq \theta_1, \theta_2 \leq 2 \pi, \ w_1(x), w_2(x) > 0, 
			\ (w_1,w_2) \in \wt{\mathcal{M}} (\alpha)  }.
		\end{aligned}		
	\]
\end{proposition}

	Let us state differences between the known results and ours, and difficulties to prove 
Theorem \ref{theorem:1.1} (here we force on Theorem \ref{theorem:1.1}).  
Since we deal with bounded potentials, the embedding $H \subset L^2(\RN) \times L^2(\RN)$ 
is not compact. According to Lions \cite{L-84-1,L-84-2}, we notice that 
the compactness of any minimizing sequence for $\ea$ is equivalent to 
the strict subadditivity of $e(\alpha)$: 
for all $(\theta_1,\theta_2) \in [0,1] \times [0,1]$ with $(\theta_1,\theta_2) \neq (1,1)$, 
	\begin{equation}\label{eq:1.4}
		e \left( \alpha_1 , \alpha_2 \right) 
		< e \left( \theta_1 \alpha_1 , \theta_2 \alpha_2 \right) 
		+ e_\infty \left( (1-\theta_1) \alpha_1 , (1-\theta_2) \alpha_2 \right)
	\end{equation}
where 
	\[
		\begin{aligned}
			e_\infty(\alpha) &:= 
			\inf_{u \in M(\alpha)} E_\infty(u), \quad 
			E_\infty(u) :=  \int_{\RN} 
			\sum_{i=1}^2 \left(  \frac{|\nabla u_i|^2}{2}  - \frac{\mu_i}{2p_i+2} |u_i|^{2p_i+2} \right) 
			- \frac{\beta}{p_3+1} |u_1 u_2 |^{p_3+1} \rd x
			.
		\end{aligned}
	\]
Here we note that when $p_3 \geq 1$, by following the argument in \cite[p.228--p.229]{L-84-2} 
(see also \cite{B-17,C-03} and Appendix \ref{section:A}), 
it is possible to prove \eqref{eq:1.4} under the assumption $V_1(x), V_2(x) \leq 0$ and 
$V_1(x), V_2(x) \to 0$ as $|x| \to \infty$. 
However, when $0<p_3<1$, it is not straightforward to exploit the argument in \cite{L-84-2} 
and we face a difficulty in proving \eqref{eq:1.4}.

	In the cases $V_1(x) \equiv 0 \equiv V_2(x)$ \cite{GJ-16} and 
$V_1(x_1,x_2,x_3) = x_1^2 + x_2^2 = V_2(x_1,x_2,x_3) $ \cite{Go-18}, 
to get \eqref{eq:1.4}, they use the coupled rearrangement due to Shibata \cite{S-17}. 
However, in our setting, this technique is not applicable 
since we do not suppose any symmetry of $V_1(x)$ and $V_2(x)$.

	In this paper, we show \eqref{eq:1.4} by combining the scheme in \cite{S14,I-14} and 
the interaction estimate in Bahri-Li \cite{BL-90}. 
This approach is also taken in \cite{IM-19} for a minimizing problem corresponding to 
a scalar equation with potential and general nonlinearity. 
Compared to the scalar case \cite{IM-19}, we have the following differences. 
First, we need to observe the decay estimate of 
nonnegative solutions to \eqref{eq:1.3}. This will be discussed in subsection \ref{section:3.1} and 
we point out that the decay properties of solutions to \eqref{eq:1.3} 
depend on the size of $p_3$, $\lambda_1$ and $\lambda_2$. 
Next, we need to prepare some inequality which is used in the interaction estimate for the term 
$\int_{\RN} u_1^{p_3+1} u_2^{p_3+1} \rd x$. 
We emphasize that when $0<p_3<1$, the term $u_1^{p_3+1} u_2^{p_3+1}$ 
is not of class $C^2$, hence, we need a careful analysis. See Lemma \ref{lemma:3.4} (ii).

	This paper is organized as follows. 
In section \ref{section:2}, we do some preparations which are necessary for proofs of 
Theorems \ref{theorem:1.1} and \ref{theorem:1.2}. 
In section \ref{section:3}, we observe the decay properties of solutions of \eqref{eq:1.3} and 
prove useful inequalities and Theorem \ref{theorem:1.1}. 
Section \ref{section:4} is devoted to a proof of Theorem \ref{theorem:1.2} and 
in Appendix \ref{section:A}, we provide a proof of \eqref{eq:1.4} when $p_3 \geq 1$ by following the arguments 
in \cite{L-84-2} (cf. \cite{B-17,C-03}).

\section{Preliminaries}
\label{section:2}

	In this section, we shall prove several results which are necessary 
for the proofs of Theorems \ref{theorem:1.1} and \ref{theorem:1.2}. 
Similar results are used in \cite{I-14,IM-19}.

We first notice that for every $b_1,b_2 \in \R$, 
the minimizing problem for $e(\alpha)$ (resp. $\tea$) is equivalent to 
	\[
		\begin{aligned}
			&\inf \left\{ E(u)  - \sum_{i=1}^2 \frac{b_i}{2} \| u_i \|_{L^2}^2 \ \bigg| \ 
			u \in M(\alpha)   \right\} = e(\alpha) - \sum_{i=1}^2 \frac{b_i}{2} \alpha_i
			,
			\\
			&\left( \text{resp.} \ \ 
			\inf \left\{ E(u)  -  \sum_{i=1}^2 \frac{b_i}{2} \| u_i \|_{L^2}^2 \ \bigg| \ 
			u \in \wt{M}(\alpha)   \right\} 
			= \tea -  \sum_{i=1}^2 \frac{b_i}{2} \alpha_i
			 \right).
		\end{aligned}
	\]
Therefore, without loss of generality, we may assume the following conditions 
instead of (V1') and (V2'): 
\begin{enumerate}
	\item[(V1)] 
		$V \in C(\RN,\R)$, $V(x) \leq 0$ for all $x \in \RN$ and $\lim_{|x| \to \infty} V(x) = 0$. 
	\item[(V2)] 
		$V_1$ satisfies (V1),  
		$V_2 \in C(\RN,\R)$, $1 = \inf_{x \in \RN} V_2(x)$ and 
		$V_2(x) \to \infty$ as $|x| \to \infty$. 
\end{enumerate}

	We set 
	\begin{equation}\label{eq:2.1}
		\la u_i, v_i \ra_{V_i} := \RE  \int_{\RN} \nabla u_i \cdot \ov{\nabla v_i} \rd x + 
		\RE \int_{\RN} V_i(x) u_i \ov{v_i} \rd x, \quad 
		\| u_i \|_{V_i}^2 := \la u_i , u_i \ra_{V_i} .
	\end{equation}
Even though the quantity $\| \cdot \|_{V_i}$ may not be a norm on $H^1$, 
this notation will be convenient in subsection \ref{section:3.3}. 
Remark that  under (V2), $\| \cdot \|_{V_2}$ is a norm on $\tH_2$ and 
$(\tH_2, \la \cdot, \cdot \ra_{V_2} )$ is a Hilbert space over $\R$.

	To prove Theorem \ref{theorem:1.1}, we begin with the following lemma:
	\begin{lemma}\label{lemma:2.1}
		Assume \textup{(p1)}. 
		\begin{enumerate}
			\item 
				Suppose \textup{(V1)}. Then $-\infty < \ea$ holds and 
				every minimizing sequence for $\ea$ is bounded in $H$. 
				If we assume \textup{(V2)} instead of \textup{(V1)}, 
				then $-\infty < \tea$ and any minimizing sequence for $\tea$ is bounded in $\tH$.
			\item 
				Under \textup{(V1)}, $\ea < 0$ holds .
			\item 
				Suppose \textup{(V1)}. Then $\ea \leq e(\gamma)  + e_\infty(\alpha-\gamma)$
				for any $\gamma = (\gamma_1,\gamma_2) \in [0,\alpha_1] \times [0,\alpha_2]$ 
				where 
					\[
						\begin{aligned}
							e_\infty(\alpha) &:= 
							\inf \Set{ E_\infty(u) | u \in M(\alpha) }, 
							\\
							E_\infty(u) &:= 	\frac{1}{2}\int_{\RN} \sum_{i=1}^2 |\nabla u_i|^2  \rd x 
								- \int_{\RN} \frac{\mu_1 |u_1|^{2p_1+2}}{2p_1 + 2} 
								+ \frac{\mu_2 |u_2|^{2p_2+2}}{2p_2 + 2}
								+ \frac{\beta |u_1|^{p_3+1} |u_2|^{p_3+1} }{p_3+1}  \rd x. 
						\end{aligned}
					\]
				If we suppose \textup{(V2)} instead of \textup{(V1)}, then 
				$\tea \leq \wt{e}( \alpha_1 - \gamma, \alpha_2 ) + e_1(\gamma)$ for all $\gamma \in [0,\alpha_1]$ where 
					\[
						\begin{aligned}
							e_1 (\gamma) 
							&:= \inf \Set{ E_1(u_1) | u_1 \in H^1(\RN), \ \| u_1 \|_{L^2}^2 = \gamma },
							\\ 
							E_1(u_1) 
							&:= \frac{1}{2} \int_{\RN} | \nabla u_1 |^2 \rd x 
							- \frac{\mu_2}{2p_2+2} \int_{\RN} |u_1|^{2p_2+2} \rd x.
						\end{aligned}
					\]
		\end{enumerate}
	\end{lemma}

	\begin{proof}
(i) 
We first prove the assertion for $\ea$. 
Let $(u_n) \subset M(\alpha)$ satisfy $E(u_n) \to \ea$. 
It suffices to show that $(\nabla u_{n,i})$ ($i=1,2$) are bounded in $L^2(\RN)$. 
By Gagliardo-Nirenberg's inequality, 
	\[
		\| u \|_{L^{2p+2}}^{2p+2} \leq 
		C \| \nabla u \|_{L^2}^{\theta(2p+2)} \| u \|_{L^2}^{(1-\theta)(2p+2)},
		\quad \theta := \left( \frac{1}{2} - \frac{1}{2p + 2} \right) N.
	\]
Thus if $p < 2/N$, then one sees that 
$(2p+2) \theta = N p < 2$. Since $V_i(x) \leq 0$ and 
	\begin{equation}\label{eq:2.2}
		\begin{aligned}
			E(u_n) &\geq  
			\frac{1}{2} \left( \| \nabla u_{n,1} \|_{L^2}^2 
			+ \| \nabla u_{n,2} \|_{L^2}^2 \right) 
			+ \frac{1}{2} \int_{\RN} V_1(x) |u_{n,1}|^2 + V_2(x) |u_{n,2} |^2 \rd x 
			\\
			&\qquad 
			- \frac{\mu_1 \| u_{n,1}\|_{L^{2p_1+2}}^{2p_1+2} }{2p_1 + 2}
			- \frac{\mu_2 \| u_{n,2}\|_{L^{2p_2+2}}^{2p_2+2} }{2p_2 + 2}
			- \frac{\beta \| u_{n,1}\|_{L^{2p_3+2}}^{p_3+1} \| u_{n,2} \|_{L^{2p_3+2}}^{p_3+1} }
			{p_3+1}
			\\
			&\geq 
			\frac{1}{2} \sum_{i=1}^2  \| \nabla u_{n,i} \|_{L^2}^2 
			- C \left( \| \nabla u_{n,1} \|_{L^2}^{N p_1 } + \| \nabla u_{n,2} \|_{L^2}^{Np_2}
			+ \| \nabla u_{n,1} \|_{L^2}^{Np_3/2} \| \nabla u_{n,2} \|_{L^2}^{Np_3/2} 
			 \right),
		\end{aligned}
	\end{equation}
we have $\ea > - \infty$ and $(\nabla u_{n,i})$ ($i=1,2$) are bounded in $L^2(\RN)$.

	On the other hand, when (V2) holds, 
as in the above, we see that $(\| \nabla u_{n,i} \|_{L^2} )$ ($i=1,2$) are bounded. 
Furthermore, from \eqref{eq:2.2}, 
it follows that $\sup_{n \geq 1} \int_{\RN} V_2(x) |u_{n,2}|^2 \rd x < \infty$, 
hence,  $(u_n)$ is bounded in $\tH$.

(ii) By (V1), for any $u \in M(\alpha)$, it is immediate to check $E(u) \leq E_\infty(u)$, 
which yields $\ea  \leq e_\infty(\alpha)$. 
According to \cite{GJ-16}, it is known that 
$e_\infty(\alpha) < 0$. Hence, we have 
$e(\alpha) < 0$.

	(iii) Since we may prove both of inequalities in a similar way, 
we only show the desired inequality for $\ea$.  For every $\e>0$ and $\gamma \in [0,\alpha_1] \times [0,\alpha_2]$, 
we may find $\varphi_\e, \psi_\e \in (C^\infty_0(\RN))^2$ so that 
	\[
		\varphi_\e \in M(\gamma), \ 
		\psi_\e \in M(\alpha - \gamma), \ 
		E(\varphi_\e) \leq e(\gamma) + \e, \ 
		E_\infty(\psi_\e) \leq e_\infty(\alpha - \gamma) + \e.
	\]
Set $u_{\e,n}(x) := \varphi_\e(x) + \psi_\e(x - n\mathbf{e}_1)$. 
Since $\varphi_\e,\psi_\e$ have compact support, 
we have $u_{\e,n} \in M(\alpha)$ for sufficiently large $n$ and 
$\ea \leq E(u_{\e,n}) = E(\varphi_\e) + E(\psi_\e(\cdot - n\mathbf{e}_1))$. 
By letting $n \to \infty$ and noting $E_\infty( u ( \cdot - y ) ) = E_\infty(u)$ for all $y \in \RN$, 
(V1) gives 
	\[
		\ea \leq E(\varphi_\e) + E( \psi_\e(\cdot - n \mathbf{e}_1 ) ) 
		\leq E(\varphi_\e) + E_\infty(\psi_\e (\cdot - n \mathbf{e}_1) ) \leq 
		e(\gamma) + e_\infty(\alpha - \gamma) + 2 \e.
	\]
Since $\e>0$ is arbitrary, the desired inequality for $\ea$ holds. 
	\end{proof}

	Next, we show that it is enough to consider a minimizing sequence 
whose components are nonnegative. 

	\begin{lemma}\label{lemma:2.2}
		Under \textup{(p1)} and \textup{(V1)}, 
		let $(u_n) \subset M(\alpha)$ satisfy $E(u_n) \to \ea$. 
		Then $|u_n|(x) := (|u_{n,1}(x)|, |u_{n,2}(x)|)$ is also 
		a minimizing sequence. Moreover, if $(|u_n|)$ has a 
		strongly convergent subsequence in $H$, so is $(u_n)$. 
		
		The same assertions hold when we suppose \textup{(V2)} instead of \textup{(V1)}.
	\end{lemma}

	\begin{proof}
For $\ea$, by \cite[Theorem 6.17]{LL-01}, we have 
$\| \nabla |u_{n,i}| \|_{L^2}^2 \leq \| \nabla u_{n,i} \|_{L^2}^2$. 
Since $|u_n| \in M(\alpha)$, one observes that 
$E(|u_n|) \leq E(u_n)$. 
Hence, $(|u_n|)$ is also a minimizing sequence.

	Next, assume that $(|u_n|)$ has a strongly convergent subsequence in $H$. 
Hereafter, we still write $(u_n)$ even after taking a subsequence. 
Thus, we suppose that $|u_n| \to (v_1,v_2)$ strongly in $H$. 
Since $(u_n)$ is bounded in $H$ by Lemma \ref{lemma:2.1}, we have 
$u_n \rightharpoonup u_0=(u_{0,1},u_{0,2})$ weakly in $H$ and 
$u_{n,i} \to u_{0,i}$ in $L^q_{\rm loc}(\RN)$ for $1 \leq q < 2^\ast = 2N/(N-2)$. 
Therefore, we infer that 
$v_i(x) = |u_{0,i}(x)|$.

	From $|u_{n,i}| \to |u_{0,i}|$ strongly in $L^2(\RN)$ with $i=1,2$, 
after taking a sequence, we may find $w_i \in L^2(\RN)$ such that 
$|u_{n,i}(x)| \leq w_i(x)$ and $|u_{n,i}(x)| \to |u_{0,i}(x)|$ for a.a. $x \in \RN$. 
The dominated convergence theorem yields 
$u_n \to u_0$ strongly in $L^2(\RN) \times L^2(\RN)$ and 
$u_n \to u_0$ strongly in $L^q(\RN) \times L^q(\RN)$ for every $ 2 \leq q < 2^\ast$. 
In particular, $u_0 \in M(\alpha)$.

	Finally, by the weak lower semicontinuity and the facts $u_0 \in M(\alpha)$ 
and $u_n \to u_0$ in $L^q(\RN) \times L^q(\RN)$ for $2 \leq q < 2^\ast$, 
we obtain 
	\[
		\ea \leq E(u_0) \leq \liminf_{n\to \infty} E(u_n) = \ea,
	\]
which implies $\| \nabla u_{n,i} \|_{L^2}^2 \to \| \nabla u_{0,i} \|_{L^2}^2$. 
Hence, by recalling $\nabla u_{n,i} \rightharpoonup \nabla u_{0,i}$ weakly in $L^2(\RN)$, 
one sees that $\nabla u_{n,i} \to \nabla u_{0,i}$ strongly in $L^2(\RN)$. 
Thus $u_n \to u_0$ strongly in $H$.

	On the other hand, for $\tea$, 
as in the above, we may prove $u_n \to u_0$ strongly in $L^q(\RN) \times L^q(\RN)$ 
for $2 \leq q < 2^\ast$. Thus, 
	\[
		\tea \leq E(u_0) \leq \liminf_{n \to \infty} E(u_n) = \tea 
	\]
and this yields $\| u_{n,2} \|_{V_2}^2 \to \| u_{0,2} \|_{V_2}^2$. 
Hence, $u_n \to u_0$ strongly in $\tH$. 
	\end{proof}

	Next, we show the existence of Palais--Smale sequence 
close to any given minimizing sequence.

	\begin{lemma}\label{lemma:2.3}
		\begin{enumerate}
			\item 
				Suppose \textup{(p1)} and \textup{(V1)}, and let 
				$(u_n) \subset M(\alpha)$ be a minimizing sequence. 
				Then there exist $(v_n) \subset M(\alpha)$ and $(\lambda_{n,i}) \subset \R$ such that 
				$(\lambda_{n,i})$ are bounded and 
					\[
						\| u_n - v_n \|_{H} \to 0, \quad 
						E'(v_n) + \lambda_{n,1} Q_1'(v_n) + \lambda_{n,2} Q'_2(v_n)  \to 0
						\quad {\rm strongly \ in} \ H^\ast
					\]
				where 
					\[
						\begin{aligned}
						E'(u)v &= \RE 
						\int_{\RN} \sum_{i=1}^2 \nabla u_i \cdot 
						\ov{\nabla v_i} + V_i(x) u_i \ov{v_i} \rd x 
						- \RE 
						\int_{\RN} \mu_1 |u_1|^{2p_1} u_1 \ov{v_1} 
						+ \mu_2 |u_2|^{2p_2} u_2 \ov{v_2} 
						 \rd x 
						\\
						& \quad - \RE 
						\int_{\RN} 
						 \beta |u_1|^{p_3-1} |u_2|^{p_3+1} u_1 \ov{v_1}
						+ \beta |u_1|^{p_3+1} |u_2|^{p_3-1} u_2 \ov{v_2} \rd x,
						\\
						Q_i(u) &:= \frac{1}{2}\int_{\RN} |u_i|^2 \rd x, \quad 
						Q_i'(u) v =  \RE  \int_{\RN} u_i \ov{v_i} \rd x.
					\end{aligned}
				\]
			\item Assume \textup{(p1)} and \textup{(V2)}. 
			Then for every minimizing sequence $(u_n) \subset \wt{M}(\alpha)$ for $\tea$, 
			there exists $(v_n) \subset \wt{M}(\alpha)$ and $( \lambda_{n,i} ) \subset \R$ such that 
			$(\lambda_{n,i})$ are bounded and 
				\[
					\| u_n - v_n \|_{H} + \| u_{n,1} - v_{n,1} \|_{V_1} \to 0, \quad 
					E'(v_n) + \lambda_{n,1} Q_1'(v_n) + \lambda_{n,2} Q_2'(v_n) \to 0 \quad 
					\textrm{strongly\ in} \ (\wt{H})^\ast.
				\]
		\end{enumerate}
	\end{lemma}

	\begin{proof}
Since a proof for $\tea$ is similar to the one for $\ea$, we only show (i). 
Note that $M(\alpha)$ is a closed set in $H$. Applying Ekeland's variational principle 
for $E$ and $(u_n)$ on $M(\alpha)$ (\cite[Theorem 4.1 and Remark 4.1]{MW-89}), 
we may find $(v_n) \subset M(\alpha)$ such that 
for $\e_n := E(u_n) - \ea \geq 0$, 
	\begin{equation}\label{eq:2.3}
		E(v_n) \leq E(u_n), \quad \| u_n - v_n \|_{H} \leq \sqrt{\e_n}, \quad 
		E(v_n) \leq E(w) + \sqrt{\e_n} \| v_n - w\|_{H} \quad 
		{\rm for\ all} \ w \in M(\alpha).
	\end{equation}
Thus, $(v_n)$ is also a minimizing sequence.

	Since 
	\begin{equation}\label{eq:2.4}
		M(\alpha) = \Set{ u \in H | Q_1(u) = \frac{\alpha_1}{2}, \ 
		Q_2(u) =  \frac{\alpha_2}{2} }, \quad 
		Q_1'(u) [(u_1,0)] = \alpha_1, \quad 
		Q_2'(u) [(0,u_2)] = \alpha_2,
	\end{equation}
we see that $Q_1'(u)$ and $Q_2'(u)$ are linearly independent for each $u \in M(\alpha)$, 
and that $M(\alpha)$ is a Hilbert manifold with codimension $2$. 
Its tangent space at $ u \in M(\alpha)$ is given by 
	\begin{equation}\label{eq:2.5}
		T_u M(\alpha) = \left\{ v \in H \ |\ Q_1'(u) v = 0 = Q_2'(u) v  \right\} 
		(= {\rm span} \{ \nabla Q_1(u), \nabla Q_2(u) \}^\perp  )
	\end{equation}
where $\nabla Q_i(u) \in H$ satisfies 
$\la \nabla Q_i(u) , v \ra_{H} = Q_i'(u) v$ for all $v \in H$. 
By \eqref{eq:2.3}, one gets 
	\[
		\left\|  E'(v_n)  \right\|_{(T_{v_n}M(\alpha) )^\ast} 
		= \sup_{ w \in T_{v_n}M(\alpha), \  \| w \|_{H} \leq 1 } 
		E'(v_n) w
		\to 0.
	\]
From \eqref{eq:2.5}, for any $\lambda_1,\lambda_2 \in \R$, we have 
	\[
		\left(E'(v_n) + \lambda_1 Q_1'(v_n) + \lambda_2 Q_2'(v_n) \right) w 
		= E'(v_n) w \quad {\rm for\ all} \ w \in T_{v_n}M(\alpha).
	\]
In addition, since $(v_n)$ is bounded in $H$ due to Lemma \ref{lemma:2.1}, 
\eqref{eq:2.4} implies that there exists $0<c_1<c_2<\infty $ such that 
$c_1 \leq \| Q_i'(v_n) \|_{H^\ast} \leq c_2$ for every $n \geq 1$. 
Thus, setting 
$\lambda_{n,i} := - E'(v_n) \nabla Q_i(v_n) / \| Q_i'(v_n) \|_{H^\ast}^2$, 
we observe that $(\lambda_{n,i})$ are bounded and 
	\[
		E'(v_n) + \lambda_{n,1} Q_1'(v_n) + \lambda_{n,2} Q_2'(v_n) 
		\to 0 \quad {\rm strongly \ in} \ H^\ast.
	\]
Thus, Lemma \ref{lemma:2.3} holds. 
	\end{proof}

\section{Proof of Theorem \ref{theorem:1.1}}
\label{section:3}

	In this section, we aim to prove Theorem \ref{theorem:1.1}. 
Before proceeding to the proof of Theorem \ref{theorem:1.1}, 
we prepare several useful lemmas. 
Throughout this section, we always assume that $V_1$ and $V_2$ satisfy (V1) and (p1) holds

\subsection{Decay estimates}
\label{section:3.1}

In this subsection, we observe the decay rate of nonnegative solutions to 
\begin{equation}\label{eq:3.1}
\left\{\begin{aligned}
-\Delta u_1 + \left( V_1(x) + \lambda_{0,1} \right) u_1 
&= \mu_1 u_1^{2p_1 + 1} + \beta u_1^{p_3} u_2^{p_3+1} 
& &{\rm in} \ \RN,\\
-\Delta u_2 + \left( V_2 (x) + \lambda_{0,2} \right) u_2 
&= \beta u_1^{p_3+1} u_2^{p_3} + \mu_2 u_2^{2p_2+1}
& &{\rm in} \ \RN
\end{aligned}\right.
\end{equation}
where $\lambda_{0,1}, \lambda_{0,2}$ are constants satisfying 
\begin{equation}\label{eq:3.2}
0< \lambda_{0,1} \leq \lambda_{0,2}. 
\end{equation}

The aim of this subsection is to prove the following decay estimates:

	\begin{lemma}\label{lemma:3.1}
		Let $u=(u_1,u_2)$ be a solution of \eqref{eq:3.1} with $u_1(x),u_2(x) \geq 0$ in $\RN$. 
		\begin{enumerate}
			\item 
				If $u_1 \not\equiv 0$, then for every $\lambda \in (0,\lambda_{0,1})$, 
				there exist $c_0>0$, which is independent of $\lambda$, and $C_\lambda>0$ such that 
				\[
					c_0 \left( 1 + |x| \right)^{-(N-1)/2} \exp \left( - \sqrt{\lambda_{0,1}} |x| \right)
					\leq u_1 (x) \leq C_\lambda \exp \left( - \sqrt{\lambda} |x| \right)
					\quad \text{for all} \ x \in \RN.
				\]
	
		\item
			If $u_1 \equiv 0$ and $u_2 \not \equiv 0$, then 
			for every $\lambda \in (0,\lambda_{0,2})$, 
			there exist $c_0>0$, which is independent of $\lambda$, 
			and $C_\lambda >0$ such that 
			\[
				c_0 \left( 1 + |x| \right)^{-(N-1)/2} \exp \left( - \sqrt{\lambda_{0,2}} |x| \right)
				\leq u_2 (x) \leq C_\lambda \exp \left( - \sqrt{\lambda} |x| \right)  
				\quad \text{for all } x \in \RN.
			\]
	
		\item
			Suppose that $u_1,u_2 \not \equiv 0$. 
			When either $p_3 \geq 1$ or else $0 < p_3 < 1$ and 
			$\lambda_{0,3} :=  (1+p_3)^2 \lambda_{0,1} / (1-p_3)^2 > \lambda_{0,2}$, 
			for every $\lambda \in (0,\lambda_{0,2})$, 
			there exist $c_0>0$, which is independent of $\lambda$, 
			and $C_\lambda >0$ such that 
			\[
				c_0 \left( 1 + |x| \right)^{-(N-1)/2} 
				\exp \left( - \sqrt{\lambda_{0,2}} |x| \right) 
				\leq u_{2}(x) 
				\leq C_\lambda \exp \left( - \sqrt{\lambda} |x| \right)
				\quad \text{for all } x \in \RN.
			\]
	
			On the other hand, 
			when $0 < p_3 < 1$ and $\lambda_{0,3} \leq \lambda_{0,2}$, 
			for each $\lambda_1 \in ( 0,  \lambda_{0,3} ) $ and 
			$ \lambda_2 \in ( \lambda_{0,3} , \infty ) $, there exist $C_{\lambda_1}, C_{\lambda_2} > 0$ 
			such that 
			\[
				C_{\lambda_2} \exp \left( - \sqrt{\lambda_2} |x| \right) 
				\leq u_2 (x) \leq 
				C_{\lambda_1} \exp \left( - \sqrt{\lambda_1} |x| \right)
				\quad \text{for all } x \in \RN. 
			\]
		\end{enumerate}
	\end{lemma}

\begin{remark}\label{remark:3.2}
	Putting $V_1(x) \equiv 0 \equiv V_2(x)$ in \eqref{eq:3.1}, 
	we observe that 
	the same statements in Lemma \ref{lemma:3.1} hold
	for solutions $w=(w_1,w_2)$ of \eqref{eq:3.3}:
	\begin{equation}\label{eq:3.3}
	\left\{\begin{aligned}
	-\Delta w_1 +  \lambda_{0,1} w_1 
	&= \mu_1 w_1^{2p_1 + 1} + \beta w_1^{p_3} w_2^{p_3+1} 
	& &{\rm in} \ \RN,\\
	-\Delta w_2 + \lambda_{0,2} w_2 
	&= \beta w_1^{p_3+1} w_2^{p_3} + \mu_2 w_2^{2p_2+1}
	& &{\rm in} \ \RN.
	\end{aligned}\right.
	\end{equation}
\end{remark}

	To show Lemma \ref{lemma:3.1}, 
we prepare one lemma which is a variant of 
\cite[Proposition 1.2]{BL-90} and \cite[Lemma II.2]{BL-97} (cf. \cite{GNN-81}) 
and also useful for the interaction estimate below: 
\begin{lemma}\label{lemma:3.3}
	Let $0 \leq f(x), g(x)$ satisfy $f, g \in C( \RN \setminus \{0\} ) \cap L^1(\RN)$ 
	and 
	\[
	\lim_{|x| \to \infty} 
	\left( 1 + |x| \right)^\alpha e^{\beta |x|} g(x) = \gamma \in [0,\infty), \quad 
	f(x) \leq C \exp \left( - (\beta + \eta) |x| \right) 
	\quad {\rm for\ each} \ |x| \geq 1
	\]
	for some $\alpha \geq 0$, $\beta \geq 0$ and $\eta > 0$. Then 
	\[
	\lim_{r \to \infty} (1+r)^\alpha e^{\beta r}
	\int_{\RN} g(r \omega - y) f(y) \rd y  
	= 
	\gamma \int_{\RN} f(y) e^{\beta \omega \cdot y} \rd y 
	\quad \text{uniformly with respect to} \ \omega \in \SN.
	\]
\end{lemma}

\begin{proof}
	From the assumptions, for every $\e>0$, choose an $R_\e>0$ so that 
	\begin{equation}\label{eq:3.4}
	|x| \geq R_\e \quad \Rightarrow \quad 
	\left|(1+|x|)^\alpha e^{\beta |x|} g(x) - \gamma \right| \leq \e. 
	\end{equation}
	We divide the integral into two parts:
	\[
	\begin{aligned}
	(1+r)^\alpha e^{\beta r} \int_{\RN} g(r \omega - y ) f(y) \rd y 
	&= \left( \int_{|r\omega - y| \leq R_\e} + \int_{|r \omega - y| > R_\e } \right) 
	(1+r)^\alpha e^{\beta r} g(r \omega - y) f(y) \rd y 
	\\
	& =: J_1(r,\omega) + J_2(r,\omega).
	\end{aligned}
	\]

	We first show that 
	\begin{equation}\label{eq:3.5}
	\lim_{r \to \infty}\sup_{\omega \in \SN} | J_1(r,\omega) | = 0. 
	\end{equation}
	It follows from the assumptions that for $r \geq R_{\e} + 1$, 
	\[
	\begin{aligned}
	|J_1(r,\omega)| 
	&\leq \int_{|z| \leq R_\e} g(z) f(r\omega - z) e^{\beta r} (1+r)^\alpha \rd z 
	\\
	&\leq 
	\int_{|z| \leq R_\e} g(z) C \exp \left( - (\beta + \eta) |r\omega - z| \right) 
	e^{\beta r} (1+r)^\alpha \rd z 
	\\
	&= C \int_{|z| \leq R_\e} g(z) \exp \left( \beta \{r - |r\omega - z|\} \right) 
	(1+r)^\alpha e^{-\eta |r\omega - z|}
	\rd z .
	\end{aligned}
	\]
	Noting $r - |r \omega - z| \leq |z|$, $g \in L^1(B_{R_\e})$ and 
	\[
	\lim_{r \to \infty}\sup_{\omega \in \SN, \, |z| \leq R_\e}  
	(1+r)^\alpha e^{-\eta |r\omega - z|} = 0,
	\]
	we observe that \eqref{eq:3.5} holds.

	Next, we consider $J_2(r,\omega)$. 
	Since 
	\[
	J_2(r,\omega) 
	= 
	\int_{|r\omega - y | > R_\e} 
	g(r \omega - y) \left( 1+ |r \omega - y | \right)^\alpha 
	e^{\beta |r \omega - y|} 
	f(y) e^{\beta ( r - |r\omega - y| )} 
	\left( \frac{1+r}{1+|r\omega -y |} \right)^\alpha \rd y
	\]
	and $r - |r \omega - y| \leq |y|$, we see from \eqref{eq:3.4} and the assumption on $f$ that 
	\begin{equation}\label{eq:3.6}
	\begin{aligned}
	&\sup_{\omega \in \SN} 
	\left| J_2(r,\omega) - \gamma \int_{|r\omega - y| > R_\e} 
	f(y) e^{\beta ( r - |r\omega - y| )} 
	\left( \frac{1+r}{1+|r\omega -y |} \right)^\alpha \rd y  \right| 
	\\
	\leq \, & 
	\e \sup_{\omega \in \SN}
	\int_{|r \omega - y | >R_\e}
	f(y) e^{\beta ( r - |r\omega - y| )} 
	\left( \frac{1+r}{1+|r\omega -y |} \right)^\alpha \rd y 
	\\
	\leq \, &
	\e \sup_{\omega \in \SN} \int_{\RN} f(y) e^{\beta |y|} 
	\left( \frac{1+|r\omega - y | + |y|}{1 + |r\omega - y |} \right)^\alpha
	\rd y 
	\\
	\leq & \, \e \int_{\RN} f(y) e^{\beta |y| } \left( 1 + |y| \right)^\alpha \rd y
	= 
	C(f) \e. 
	\end{aligned}
	\end{equation}

	Next, by 
	\[
	\begin{aligned}
	\beta (r - |r\omega -y|) 
	= \beta r \left[ 1 - \left( 1 - 2 \frac{\omega\cdot y}{r} 
	+ \frac{|y|^2}{r^2}  \right)^{1/2} \right] 
	&= \beta r \left[ \frac{\omega \cdot y}{r} + O \left( \frac{|y|^2}{r^2} \right) \right]
	\\
	&= \beta \omega \cdot y + O \left( \frac{|y|^2}{r} \right), 
	\end{aligned}
	\]
	we can choose a $\delta_\e>0$ so that 
	\[
	|y| \leq \delta_\e \sqrt{r} \quad 
	\Rightarrow \quad 
	\sup_{\omega \in \SN} 
	\left|e^{\beta (r - |r\omega -y| ) } - e^{\beta \omega \cdot y} \right| 
	< \e \sup_{\omega \in \SN} e^{\beta \omega \cdot y} 
	= \e e^{\beta |y|} .
	\]
	Set 
	\[
	[|r\omega - y| > R_\e] \cap [|y| \leq \delta_\e \sqrt{r}] := I_1(r,\omega), \quad 
	[|r\omega - y| > R_\e] \cap [ |y| > \delta_\e \sqrt{r} ] := I_2(r,\omega). 
	\]
	Noting 
	\[
		\begin{aligned}
			&\lim_{r \to \infty} \sup_{\omega \in \SN, \ y \in I_1(r,\omega)} 
			\left( \frac{1+r}{1+|r\omega -y|} \right)^\alpha = 1, 
			\\
			&\lim_{r \to \infty} \sup_{\omega \in \SN} 
			\int_{[|y| > \delta_\e \sqrt{r}] \, \cup \, [ |r \omega - y | \leq R_\e ] } 
			f(y) 
			e^{\beta |y|} (1+|y|)^\alpha \rd y = 0,
		\end{aligned}
	\]
	we obtain 
	\begin{equation}\label{eq:3.7}
	\begin{aligned}
	& \limsup_{r \to \infty} \sup_{\omega \in \SN} 
	\left| \int_{I_1(r,\omega)} f(y) e^{\beta ( r - |r\omega - y| )}
	\left( \frac{1+r}{1+|r\omega-y|} \right)^\alpha \rd y
	- \int_{\RN} f(y) e^{\beta \omega \cdot y} \rd y  \right|
	\\
	= & \, 
	\limsup_{r \to \infty} \sup_{\omega \in \SN} 
	\left| \int_{I_1(r,\omega)} f(y) e^{\beta ( r - |r\omega - y| )} \rd y
	- \int_{I_1(r,\omega)} f(y) e^{\beta \omega \cdot y}  \rd y  \right| 
	\\
	\leq & \,
	\e \int_{\RN} f(y) e^{\beta |y|} \rd y = C(f) \e.
	\end{aligned}
	\end{equation}

	On the other hand, from $r - |r \omega - y | \leq |y|$, 
	it follows that 
	\begin{equation}\label{eq:3.8}
	\begin{aligned}
	&\limsup_{r \to \infty} \sup_{\omega \in \SN} 
	\int_{I_2(r,\omega)} f(y) e^{\beta (r - |r\omega -y|) } 
	\left( \frac{1+r}{1+ |r \omega - y|} \right)^\alpha \rd y 
	\\
	\leq \, & 
	\limsup_{r \to \infty} \sup_{\omega \in \SN} 
	\int_{I_2(r,\omega)} f(y) e^{\beta |y|} 
	\left( \frac{1+ |r\omega -y | + |y|}{1+|r \omega - y |} \right)^\alpha 
	\rd y 
	\\
	\leq \, & 
	\limsup_{r \to \infty} \sup_{\omega \in \SN} 
	\int_{I_2(r,\omega)} f(y) e^{\beta |y|} 
	\left( 1 + |y| \right)^\alpha 
	\rd y 
	= 0.
	\end{aligned}
	\end{equation}

	Now from \eqref{eq:3.5} through \eqref{eq:3.8}, we obtain 
	\[
	\limsup_{r \to \infty} 
	\sup_{\omega \in \SN} 
	\left| (1+r)^\alpha e^{\beta r} \int_{\RN} g(r \omega - y ) f(y) \rd y 
	- \gamma \int_{\RN} f(y) e^{\beta \omega \cdot y} \rd y \right| 
	\leq C(f) \e.
	\]
	Since $\e>0$ is arbitrary, Lemma \ref{lemma:3.3} holds. 
\end{proof}

Now we prove Lemma \ref{lemma:3.1}.

\begin{proof}[Proof of Lemma \ref{lemma:3.1}]
	We first remark that by the strong maximum principle, 
	if $u_j(x) \not \equiv 0$, then $u_j(x) > 0$ holds for each $x \in \RN$. 
	We divide our proof into several steps.

	\medskip
	
	\noindent
	\textbf{Step 1:} \textsl{Proof of (ii).}
	
	\medskip

	To this end, it suffices to observe a positive solution of 
	\[
	-\Delta u + (\lambda + V(x)) u = \mu u^{2p+1} \quad {\rm in} \ \RN, 
	\quad u \in H^1(\RN)
	\]
	where $\lambda,\mu > 0$, $V(x) \leq 0$ and $V(x) \to 0$ as $|x| \to \infty$. 
	We shall prove that for every $\nu \in (0,\lambda)$, 
	there exist $c_0>0$, which is independent of $\nu$, and 
	$C_\nu > 0$ such that 
	\begin{equation}\label{eq:3.9}
		c_0 \left( 1 + |x| \right)^{-(N-1)/2} \exp \left( - \sqrt{\lambda} |x| \right) 
		\leq u (x) \leq 
		C_\nu \exp \left( - \sqrt{\nu} |x| \right) 
		\quad \text{for all } x \in \RN. 
	\end{equation}
	
	Let $0 <\nu  < \lambda$. 
	Since $V(x), u(x) \to 0 $ as $|x| \to \infty$, 
	we can choose an $R_\nu>0$ such that 
	\[
	|x| \geq R_\nu \quad \Rightarrow \quad 
	u(x) \leq 1, \quad 
	\lambda + V(x) - \mu u^{2p}(x) \geq \frac{\lambda + \nu}{2}.
	\]
	Therefore, $u$ satisfies 
	\[
	-\Delta u + \frac{\lambda + \nu}{2} u \leq 0 \quad 
	{\rm in} \ [|x| > R_\nu]. 
	\]

	On the other hand, set 
	\[
	\psi_\nu(r) := \exp \left( - \sqrt{\nu} (r - R_\nu) \right).
	\]
	Noting 
	\[
	|\psi_\nu'(r) | = \sqrt{\nu} \psi_\nu(r), \quad 
	\psi_\nu''(r) = \nu \psi_\nu (r),
	\]
	we obtain 
	\[
	\begin{aligned}
	-\Delta ( \psi_\nu(|x|) ) + \frac{\lambda+\nu}{2} \psi_{\nu} (|x|) 
	&= - \psi_{\nu}''(|x|) - \frac{N-1}{|x|} \psi_{\nu}'(|x|) 
	+ \frac{\lambda + \nu}{2} \psi_{\nu} (|x|) 
	\\
	& \geq - \frac{N-1}{|x|} \sqrt{\nu} \psi_\nu(|x|) + \frac{\lambda - \nu}{2} 
	\psi_\nu (|x|).
	\end{aligned}
	\]
	Thus, enlarging $R_\nu$ if necessary, we may assume that 
	\[
	\frac{N-1}{R_\nu} \sqrt{\nu} \leq \frac{\lambda - \nu}{2}.
	\]
	From this, it follows that 
	\[
	-\Delta u + \frac{\lambda + \nu}{2} u \leq 0 \leq 
	- \Delta \psi_\nu(|x|) + \frac{\lambda + \nu}{2} \psi_{\nu} \quad 
	{\rm in} \ [|x| > R_\nu].
	\]
	Noting $u(x) \leq 1 = \psi_{\nu } (|x|)$ on $\partial B_{R_\nu}(0)$ and 
	$\psi_\nu(x), u(x) \to 0$ as $|x| \to \infty$, we find that 
	$\psi_\nu - u$ does not take a negative minimum in $[|x| > R_\nu]$. 
	Hence, $u(x) \leq \psi_{\nu} (x)$ in $[|x| > R_\nu]$ and this implies 
	\[
	u(x) \leq C_\nu \exp \left( - \sqrt{\nu} |x| \right) 
	\quad \text{for all } x \in  \RN. 
	\]
Since $\nu \in (0,\lambda)$ is arbitrary, we have the upper estimates in \eqref{eq:3.9}.

	On the other hand, let $G_{\lambda}$ be the Green function for $-\Delta + \lambda$. 
	It is known that 
	\begin{equation}\label{eq:3.10}
	\begin{aligned}
	&
	G_{\lambda} \in C(\RN \setminus \{0\}) \cap L^1(\RN), 
	\quad G_\lambda(x) > 0 \quad  {\rm in} \ \RN \setminus \{0\}, 
	\\
	&
	\lim_{|x| \to \infty} G_{\lambda}(x) |x|^{(N-1)/2} e^{\sqrt{\lambda} |x| } 
	= \gamma_0 \in (0,\infty). 
	\end{aligned}
	\end{equation}
	For instance, see \cite{GNN-81}. Since $V(x) \leq 0$ and $u>0$, 
	we have 
	\[
	-\Delta G_{\lambda} + \lambda G_{\lambda} = 0 
	\leq - V(x) u + \mu u^{2p+1} = -\Delta u + \lambda u
	\quad {\rm in} \ [|x|>1]. 
	\]
	Noting $G_\lambda(x), u(x) \to 0$ as $|x| \to \infty$ and choosing $c_1>0$ so that 
	$c_1 G_{\lambda} (x) \leq u(x)$ for any $|x| = 1$, we see that 
	$u(x) - c_1 G_{\lambda} (x) $ does not have negative minimum on $[|x| > 1]$. 
	Thus $u - c_1 G_{\lambda}  \geq 0$ in $[|x| \geq 1]$ and 
	it follows from \eqref{eq:3.10} that 
	\[
	c_0 \left( 1+|x| \right)^{-(N-1)/2} \exp \left( - \sqrt{\lambda} |x| \right) 
	\leq u(x) \quad {\rm for \ all} \ x \in \RN. 
	\]
	Hence, \eqref{eq:3.9} and Step 1 hold.

	\medskip

	Next, we treat (i) and (iii). Note that when $u_{2} \equiv 0$, 
	(i) is derived from \eqref{eq:3.9}. 
	Since the condition $u_1,u_2>0$ is assumed in (iii), 
	in what follows, we may suppose 
	$u_1,u_2 > 0$ without loss of generality.

	\medskip
	
	\noindent
	\textbf{Step 2:} \textsl{Proof of (i). }
	
	\medskip

	Set $u(x) := u_1(x) + u_2(x) $. It follows from \eqref{eq:3.1} that 
	\[
	\begin{aligned}
	&-\Delta u + \lambda_{0,1} u 
	+ \left\{ V_1(x) u_1 + (V_2(x) + \lambda_{0,2} - \lambda_{0,1} ) u_2 \right\}
	\\
	= & \, 
	\mu_1 u_1^{2p_1 + 1} + \mu_2 u_2^{2p_2+1} 
	+ \beta \left( u_1^{p_3} u_2^{p_3+1} + u_1^{p_3+1} u_2^{p_3} \right).
	\end{aligned}
	\]
	Recalling $V_j(x) \leq 0$, $u_j(x) > 0$ and \eqref{eq:3.2}, we obtain 
	\[
	-\Delta u + \lambda_{0,1} u + 2 V(x) u 
	\leq 
	\mu_1 u^{2p_1+1} + \mu_2 u^{2p_2+1} 
	+ 2 \beta u^{2p_3+1} \quad {\rm in} \ \RN
	\]
	where $V(x) := V_1(x) + V_2(x)$. Hence, for $u$, 
	we can apply the above argument and show that 
for each $\lambda \in (0,\lambda_{0,1})$ there exists a $C_\lambda > 0$ so that 
	\begin{equation}\label{eq:3.11}
		u_1(x) + u_2(x) = u(x) \leq C_\lambda \exp \left( - \sqrt{\lambda} |x| \right) 
		\quad \text{for all } x \in \RN.
	\end{equation}

	On the other hand, since 
	\[
	-\Delta u_1 + \lambda_{0,1} u_1 
	= - V(x) u_1 + \mu_1 u_1^{2p_1+1} 
	+ \beta u_1^{p_3} u_2^{p_3+1} 
	\geq 0 = - \Delta G_{\lambda_{0,1}} + \lambda_{0,1} G_{\lambda_{0,1}} 
	\quad {\rm in} \ [|x|>1],
	\]
as in the above, we may find a $c_0>0$ such that 
	\[
	c_0 \left( 1 + |x| \right)^{-(N-1)/2} \exp \left( - \sqrt{\lambda_{0,1}} |x| \right) 
	\leq u_1(x) \quad {\rm for \ each} \ x \in \RN.
	\]
Hence, (i) holds.

	\medskip
	
	\noindent
	\textbf{Step 3:} \textsl{Proof of  (iii) when either $p_3 \geq 1$ or else
		$0<p_3<1$ and $\lambda_{0,3} = (1+p_3)^2 \lambda_{0,1} / (1-p_3)^2 > \lambda_{0,2}$.}
	
	\medskip
	
	When $1 \leq p_3$, notice that $u_{2}$ satisfies 
	\[
	-\Delta u_2 + \left( \lambda_{0,2} + V_2(x) 
	- \mu_2 u_2^{2p_2} - \beta u_1^{p_3+1} u_2^{p_3-1} \right) u_2 = 0 
	\quad {\rm in} \ \RN. 
	\]
	Since $V_2(x),u_1(x),u_2(x) \to 0$ as $|x| \to \infty$, 
	we may use the argument in Step 1 to prove that 
	for each $\lambda \in (0,\lambda_{0,2})$ we find a $C_\lambda >0$ with 
	\[
		u_2(x) \leq C_\lambda \exp \left( - \sqrt{\lambda} |x| \right) 
		\quad \text{for all }  x \in \RN. 
	\]

	On the other hand, from 
	\[
	-\Delta u_2 + \lambda_{0,2} u_2 
	= - V_2(x) u_2 + \mu_2 u_2^{2p_2+1} + \beta u_1^{p_3+1} u_2^{p_3} 
	\geq 0 = - \Delta G_{\lambda_{0,2}} + \lambda_{0,2} G_{\lambda_{0,2}}
	\quad {\rm in} \ [|x|>1],
	\]
	it follows that for some $c_0>0$,
	\[
	c_0 \left( 1 + |x| \right)^{-(N-1)/2} \exp \left( - \sqrt{\lambda_{0,2}} |x| \right) 
	\leq u_2(x) \quad {\rm for\ each} \ x \in \RN. 
	\]
	Hence, when $p_3 \geq 1$, (iii) holds.

	Next, suppose $0<p_3<1$ and $\lambda_{0,3} > \lambda_{0,2}$. 
	Since the lower estimate for $u_2$ can be shown as in the above, 
	we only prove the upper estimate. 
	Set 
	\[
	f_0(x) := \mu_2 u_2^{2p_2+1} (x) 
	+ \beta u_1^{p_3+1}(x) u_2^{p_3} (x). 
	\]
	By \eqref{eq:3.11}, we may find $\xi_1 > \lambda_{0,1}$ such that 
	\begin{equation}\label{eq:3.12}
	f_0(x) \leq C_0 \exp \left( - \sqrt{\xi_1} |x| \right) \quad {\rm for\ all} \ x \in \RN. 
	\end{equation}
	Let $\nu>0$ and set $\varphi_\nu(x) := C \exp ( - \sqrt{\nu} |x| )$. Then we have 
	\[
	-\Delta \varphi_\nu +\left( \lambda_{0,2} + V_2(x) \right) \varphi_\nu 
	\geq \left( \lambda_{0,2} + V_2(x) - \nu - \frac{N-1}{|x|} \sqrt{\nu} \right) 
	\varphi_\nu. 
	\]
	Therefore, if $\nu < \min \{ \lambda_{0,2} , \xi_1 \}$, 
	then we may find an $R_\nu >0$ such that 
	\begin{equation}\label{eq:3.13}
	-\Delta \varphi_\nu +\left( \lambda_{0,2} + V_2(x) \right) \varphi_\nu 
	\geq f_0(x) = - \Delta u_2 + \left( \lambda_{0,2} + V_2(x) \right) u_2 
	\quad {\rm in} \ [|x|>R_\nu]. 
	\end{equation}
	Hence, for all $\nu \in (0,  \min\{ \lambda_{0,2}, \xi_1 \} )$, 
	\begin{equation}\label{eq:3.14}
		u_2 (x) \leq C_\nu \exp \left( - \sqrt{\nu} |x| \right) \quad 
		{\rm for\ every} \ x \in  \RN. 
	\end{equation}

	If $\xi_1 \geq \lambda_{0,2}$, then \eqref{eq:3.14} is the desired estimate 
	and we complete the proof.  
	When $\xi_1 < \lambda_{0,2}$, for any $\e \in (0,\lambda_{0,1}/2)$, set 
	\[
	\sqrt{\xi_{2,\e}} := \min \left\{ (2p_2+1) \sqrt{ \xi_1-\e}, \ 
	(p_3+1)\sqrt{\lambda_{0,1} - \e} + p_3 \sqrt{\xi_1-\e} \right\}.
	\]
	By the definition of $f_0$ and \eqref{eq:3.14}, we have 
	\[
	f_0(x) \leq C_\e \exp \left( -\sqrt{\xi_{2,\e}} |x| \right) \quad 
	{\rm for \ all}\ x \in \RN.
	\]
	Repeating the argument above, 
	for each $0<\e < \lambda_{0,1}/2$ and $0<\nu < \min \{ \lambda_{0,2} , \xi_{2,\e} \}$, 
	we obtain \eqref{eq:3.13} and 
	\[
	u_2(x) \leq C_\nu \exp \left( - \sqrt{\nu} |x| \right) \quad 
	{\rm for\ every} \ x \in  \RN. 
	\]

	When $\xi_{2,\e} \geq \lambda_{0,2}$ for some $\e \in (0,\lambda_{0,1}/2)$, 
	we complete the proof. 
	If $\xi_{2,\e} < \lambda_{0,2}$ for any $\e \in (0,\lambda_{0,1}/2)$, define 
	\[
	\sqrt{\xi_{3,\e}} := \min \left\{ (2p_2+1) \sqrt{ \xi_{2,\e}}, \ 
	(p_3+1)\sqrt{\lambda_{0,1} - \e} + p_3 \sqrt{\xi_{2,\e}} \right\}.
	\]
	Then from 
	\[
	f_0(x) \leq C_\e \exp \left( - \sqrt{\xi_{3,\e}} |x| \right) \quad 
	{\rm for\ all} \ x \in \RN,
	\]
	it follows that for each $0<\e < \lambda_{0,1}/2$ and $0< \nu < \min\{ \lambda_{0,2} , \xi_{3,\e} \}$, 
	\[
	u_2 (x) \leq C_\nu \exp \left( - \sqrt{\nu} |x| \right) \quad 
	{\rm for\ every} \ x \in  \RN. 
	\]

	By repeating the same procedure, 
	if $\xi_{n,\e} < \lambda_{0,2}$ holds for all $\e \in (0,\lambda_{0,1}/2)$, then define 
	\begin{equation}\label{eq:3.15}
	\sqrt{\xi_{n+1,\e}} := \min \left\{ (2p_2+1) \sqrt{ \xi_{n,\e}}, \ 
	(p_3+1)\sqrt{\lambda_{0,1} - \e} + p_3 \sqrt{\xi_{n,\e}} \right\}.
	\end{equation}
	To prove (iii), it suffices to show that there are $\e_0 \in (0,\lambda_{0,1}/2)$ 
	and $n_0$ such that $\xi_{n_0,\e_0} \geq \lambda_{0,2}$. 
	If this is not true, then 
	\begin{equation}\label{eq:3.16}
	\xi_{n,\e} < \lambda_{0,2} \quad {\rm for\ every} \ 
	(\e,n) \in \left( 0, \frac{\lambda_{0,1}}{2} \right) \times \N. 
	\end{equation}

	By the assumption $(1+p_3)^2 \lambda_{0,1} / (1-p_3)^2 = \lambda_{0,3} > \lambda_{0,2}$, 
	we may select an $\e_0 \in (0,\lambda_{0,1}/2)$ so that 
	\begin{equation}\label{eq:3.17}
	\frac{(1+p_3)^2}{(1-p_3)^2} \left( \lambda_{0,1} - \e_0 \right) > \lambda_{0,2}. 
	\end{equation}
	Since $0<p_3<1$, $\sqrt{s} < (2p_2+1) \sqrt{s}$ for $s >0$ and 
	\begin{equation}\label{eq:3.18}
	\sqrt{s} < (p_3+1) \sqrt{\lambda_{0,1}-\e_0} + p_3 \sqrt{s} \quad 
	\Leftrightarrow \quad 
	s < \frac{(1+p_3)^2}{(1-p_3)^2} \left( \lambda_{0,1} - \e_0 \right),
	\end{equation}
it follows from \eqref{eq:3.15}--\eqref{eq:3.18} 
with $s = \xi_{n,\e}$ that $\xi_{n,\e_0} < \xi_{n+1,\e_0} < \lambda_{0,2}$.

	Let $\xi_{\infty,\e_0} := \lim_{n\to\infty} \xi_{n,\e_0} \leq \lambda_{0,2}$. Then, 
	\[
	\sqrt{\xi_{\infty,\e_0}} \leq (p_3+1) \sqrt{\lambda_{0,1} - \e_0 } + p_3 \sqrt{\xi_{\infty,\e_0}}. 
	\]
	If $\sqrt{\xi_{\infty,\e_0}} < (p_3+1) \sqrt{\lambda_{0,1} - \e_0} + p_3 \sqrt{\xi_{\infty,\e_0}}$, 
then for sufficiently large $n$, it follows that 
	\[
	\sqrt{\xi_{\infty,\e_0}} < (p_3+1) \sqrt{\lambda_{0,1}-\e_0} + p_3 \sqrt{\xi_{n,\e_0}}, \quad 
	\sqrt{\xi_{\infty,\e_0}} < (2p_1+1) \sqrt{\xi_{n,\e_0}}.
	\]
	From the definition of $\xi_{n+1,\e_0}$, we have a contradiction: 
	$\sqrt{\xi_{\infty,\e_0}} <  \sqrt{\xi_{n+1,\e_0}} $.

	Now, we get 
	\[
	\sqrt{\xi_{\infty,\e_0}} = (p_3+1) \sqrt{\lambda_{0,1}-\e_0} 
	+ p_3 \sqrt{\xi_{\infty,\e_0}}.
	\]
	However, by \eqref{eq:3.17} and $\xi_{\infty,\e_0} \leq \lambda_{0,2}$, we have a contradiction:
	$\xi_{\infty,\e_0} = (1+p_3)^2(\lambda_{0,1}-\e_0) / (1-p_3)^2 > \lambda_{0,2}$. 
	Therefore Step 3 holds. 
	
	\medskip
	
	\noindent
	\textbf{Step 4:} \textsl{Proof of  (iii) when $0<p_3<1$ and 
		$(1+p_3)^2 \lambda_{0,1} / (1-p_3)^2 = \lambda_{0,3} \leq \lambda_{0,2}$.}
	
	\medskip

	We first consider the upper estimate for $u_2$. 
	Remark that $\lambda_{0,1} < \lambda_{0,2}$ holds in this case. 
	Hence, at the estimate \eqref{eq:3.12}, we may suppose 
	\[
	\lambda_{0,1} < \xi_1 <  \frac{(1+p_3)^2  }{(1-p_3)^2} \lambda_{0,1}
	\leq \lambda_{0,2}
	\]
	and fix an $\e_1 \in (0,\lambda_{0,1}/2)$ so that 
	\begin{equation}\label{eq:3.19}
	\lambda_{0,1} < \xi_1 < \frac{(1+p_3)^2  }{(1-p_3)^2} 
	\left( \lambda_{0,1} - \e_1 \right). 
	\end{equation}
	We argue as in Step 3 and define $(\xi_{n,\e})$ by \eqref{eq:3.15} with $\xi_{1,\e} := \xi_1 - \e$ 
	for $\e \in (0,\e_1)$, that is, 
	\[
		\sqrt{\xi_{n+1,\e}} := \min \left\{ (2p_2+1) \sqrt{ \xi_{n,\e}}, \ 
		(p_3+1)\sqrt{\lambda_{0,1} - \e} + p_3 \sqrt{\xi_{n,\e}} \right\} 
		\quad \text{for } n \geq 1.
	\]
	Recalling \eqref{eq:3.18} and \eqref{eq:3.19}, inductively, we can show that for each $n \geq 1$, 
	\[
		\begin{aligned}
			\sqrt{\xi_{n+1,\e}} \leq (p_3+1)\sqrt{\lambda_{0,1} - \e} + p_3 \sqrt{\xi_{n,\e}} 
			&\leq (p_3+1)\sqrt{\lambda_{0,1} - \e} + p_3 \sqrt{\xi_{\infty,\e}} 
			= \sqrt{\xi_{\infty,\e}},
			\\
			\xi_{n,\e} &\leq \xi_{n+1,\e}
		\end{aligned}
	\]
	where 
		\[
			\xi_{\infty,\e} := \frac{(1+p_3)^2  }{(1-p_3)^2} 
			\left( \lambda_{0,1} - \e \right).
		\]
	As before, it is not difficult to check $\xi_{n,\e} \to \xi_{\infty,\e}$. 
	Furthermore, $\xi_{\infty,\e } \to \lambda_{0,3}$ as $\e \to 0$. 
	Since for each $0< \e < \e_1$ and $0< \nu < \min \{ \xi_{n+1,\e} , \lambda_{0,2} \}$, 
	one finds a $C_\nu>0$ so that 
	\[
		u_2(x) \leq C_\nu \exp \left( - \sqrt{\nu} |x| \right) 
		\quad {\rm for\ each\ } x \in  \RN,
	\]
we may infer that for each $0< \nu < \lambda_{0,3}$ there exists a $C_\nu>0$ with 
	\[
		u_2 (x) \leq C_\nu \exp \left( - \sqrt{\nu} |x| \right) 
		\quad {\rm for\ all} \ x \in \RN. 
	\]

	Next, we show the lower estimate for $u_2$. 
	First, as in Step 3, we have 
	\begin{equation}\label{eq:3.20}
	c_0 \left( 1 + |x| \right)^{-(N-1)/2} \exp \left( -\sqrt{\lambda_{0,2}} |x| \right) 
	\leq u_2(x).
	\end{equation}
	Thus, if 
	\[
		\lambda_{0,3} = \frac{(1+p_3)^2}{(1-p_3)^2} \lambda_{0,1} = \lambda_{0,2},
	\]
	then our assertion clearly holds. Therefore, we may suppose $\lambda_{0,3} < \lambda_{0,2}$. 
	Set
	\[
	\sqrt{\zeta_1} := (p_3+1) \sqrt{\lambda_{0,1}} + p_3 \sqrt{\lambda_{0,2}}.
	\]
	By \eqref{eq:3.18} with $\e_0 = 0$, we have 
	\begin{equation}\label{eq:3.21}
	\sqrt{\lambda_{0,3}} = (p_3+1) \sqrt{\lambda_{0,1}} + p_3 \sqrt{\lambda_{0,3}} 
	< (p_3+1) \sqrt{\lambda_{0,1}} + p_3 \sqrt{\lambda_{0,2}} = \sqrt{\zeta_1}
	< \sqrt{\lambda_{0,2}}. 
	\end{equation}
	Then by (i) and \eqref{eq:3.20}, for some $\alpha_1>0$ and $\tilde{c}_1>0$, 
	\[
	f_0(x) \geq \beta u_1^{p_3+1}(x) u_2^{p_3}(x) 
	\geq \tilde{c}_1 (1+|x|)^{-\alpha_1} \exp \left( - \sqrt{\zeta_1} |x| \right). 
	\]
	Noting 
	\[
	-\Delta u_2 + \lambda_{0,2} u_2 = - V_2 (x) u_2 
	+ f_0(x) \quad {\rm in} \ \RN
	\]
	and $V_2 \leq 0$ and recalling the Green function $G_{\lambda_{0,2}}$, we have 
	\[
	\begin{aligned}
	u_2(x) &= \int_{\RN} G_{\lambda_{0,2}} (x-y) 
	\left( - V_2(y) u_2(y) + f_0(y) \right) \rd y 
	\\
	&\geq \tilde{c}_1 \int_{\RN} G_{\lambda_{0,2}} (x-y) 
	(1+|y|)^{-\alpha_1} \exp \left( - \sqrt{\zeta_1} |y| \right) \rd y .
	\end{aligned}
	\]
	By \eqref{eq:3.10}, \eqref{eq:3.21} and applying Lemma \ref{lemma:3.3}, one sees that 
	\[
	\begin{aligned}
	&(1+r)^{\alpha_1} \exp \left( \sqrt{\zeta_1} r \right) 
	\int_{\RN} G_{\lambda_{0,2}} (r \omega -y) 
	(1+|y|)^{-\alpha_1} \exp \left( - \sqrt{\zeta_1} |y| \right) \rd y 
	\\
	\to & \,  
	 \int_{\RN} G_{\lambda_{0,2}} (y) e^{\sqrt{\zeta_1} \omega \cdot y} \rd y > 0
	\quad \text{unifromly with respect to $\omega \in \SN$.} 
	\end{aligned}
	\]
	Hence, for some $c_1 > 0$, 
	\[
	u_2(x) \geq c_1 (1+|x|)^{-\alpha_1} \exp \left( - \sqrt{\zeta_1} |x| \right) 
	\quad {\rm for\ each} \ x \in \RN. 
	\]

	Next, set 
	\[
	\sqrt{\zeta_2} := (p_3+1) \sqrt{\lambda_{0,1}} + p_3 \sqrt{\zeta_1}.
	\]
	Then arguing as in the above, we have 
	\[
	\lambda_{0,3} < \zeta_2 < \zeta_1
	\]
	and for some $\tilde{c}_2>0$ and $\alpha_2 > 0$, 
	\[
	f_0(x) \geq \beta u_1^{p_3+1}(x) u_2^{p_3} (x) 
	\geq 
	\tilde{c}_2 (1+|x|)^{-\alpha_2} \exp \left( -\sqrt{\zeta_2} |x| \right).
	\]
	From this and Lemma \ref{lemma:3.3}, for some $c_2 > 0$, it follows that 
	\[
	u_2(x) \geq c_2 \left( 1 + |x| \right)^{-\alpha_2} \exp \left(  -\sqrt{\zeta_2} |x| \right).
	\]
	Inductively, we define 
	\[
	\zeta_{n+1} := (p_3+1) \sqrt{\lambda_{0,1}} + p_3 \sqrt{\zeta_n}.
	\]
	Then we can check that 
	\[
	\begin{aligned}
	&\lambda_{0,3} < \zeta_{n+1} < \zeta_n, \quad 
	f_0(x) \geq \tilde{c}_{n+1} (1+|x|)^{-\alpha_{n+1}} 
	\exp \left( - \sqrt{\zeta_{n+1}} |x| \right),
	\\
	&u_2 (x) \geq c_{n+1} \left( 1 + |x| \right)^{-\alpha_{n+1}} \exp \left( - \sqrt{\zeta_{n+1}} |x| \right). 
	\end{aligned}
	\]
	Finally, since $\lambda_{0,3} < \zeta_{n+1} < \zeta_n$, it is easily seen that 
	$\zeta_n \to \lambda_{0,3}$. 
	Hence, the desired lower estimate holds and we complete the proof. 
\end{proof}

\subsection{Inequality for the interaction estimate}
\label{section:3.2}

In this subsection, we prepare some inequalities for the interaction estimate:
\begin{lemma}\label{lemma:3.4}
	Let $p>0$. 
	
	\noindent
	\emph{(i)} There exists a $C_p>0$ such that 
	for any $a, b \geq 0$, 
	\[
	(a+b)^{2p+2} \geq a^{2p+2} + b^{2p+2} 
	+ \left( 2p + 2 \right) \left( a^{2p+1} b + a b^{2p+1} \right) 
	- C_p a^{p+1} b^{p+1} .
	\]
	\emph{(ii)} Let $\eta \in (0,p)$. Then there exists a $C_{p,\eta}>0$ such that 
	for any $a_i,b_i \geq 0$, 
	\[
	\begin{aligned}
	(a_1+b_1)^{p+1}(a_2+b_2)^{p+1} &\geq 
	a_1^{p+1}a_2^{p+1} + b_1^{p+1} b_2^{p+1} 
	+ \left( p + 1 \right) 
	\left( a_1^p  a_2^{p+1} b_1 + a_1^{p+1} a_2^p b_2 
	+ a_2 b_1^{p+1} b_2^p \right) 
	\\
	& \quad 
	- C_{p,\eta}
	\left( a_1^{p-\eta}  a_2^{p+1} b_1^{1+\eta}
	+ a_1^{1+\eta}  b_2^{p+1} b_1^{p-\eta} \right).
	\end{aligned}
	\]
\end{lemma}

\begin{proof}
	The assertion (i) is proved in \cite[Lemma 2.1]{BL-90}. For the assertion (ii), 
it is enough to consider $b_1,b_2>0$. 
Dividing by $b_1^{p+1} b_2^{p+1}$ and 
setting $x = a_1/b_1$ and $y = a_2/b_2$, 
we shall show that for each $x,y \geq 0$ 
	\[
	\begin{aligned}
	0 \leq f(x,y) 
	&: = (x+1)^{p+1}(y+1)^{p+1} 
	-x^{p+1}y^{p+1} - 1			
	\\
	& \qquad  - \left( p + 1 \right) \left( x^p y^{p+1}  + x^{p+1} y^p + y  \right) 
	+ C_{p,\eta} \left( x^{p-\eta} y^{p+1} + x^{1+\eta} \right).
	\end{aligned}
	\]
	In what follows, we may assume $C_{p,\eta} \geq 1$ and divide our proof into three steps.

	\medskip
	
	\noindent
	\textbf{Step 1:} \textsl{There exists an $x_1>0$ such that $f(x,y) \geq 0$ for all 
		$(x,y) \in [0,x_1] \times [0,\infty)$. }
	
	\medskip

	By simple computations, we obtain 
	\[
	\begin{aligned}
	\frac{\partial f}{\partial x} 
	&= (p+1)(x+1)^p (y+1)^{p+1} - (p+1) x^p y^{p+1} 
	-  (p+1) \left[  p x^{p-1} y^{p+1} + (p+1) x^p y^p \right]
	\\
	& \qquad 
	+  C_{p,\eta} \left[ (p-\eta) x^{p-\eta-1} y^{p+1} + (1+\eta) x^\eta \right]
	\\
	&=
	(p+1) \left[ (x+1)^p (y+1)^{p+1} - x^p y^{p+1} - (p+1) x^p y^p \right] 
	\\
	& \qquad + x^{p-\eta-1} y^{p+1} 
	\left[ (p-\eta) C_{p,\eta} - p(p+1) x^\eta \right] +  C_{p,\eta} (1+\eta) x^\eta.
	\end{aligned}
	\]
	Let 
	\[
	x^p \leq \frac{1}{2(p+1)} < \frac{1}{2}. 
	\]
	From $\max\{y^p, y^{p+1} \} \leq (1+y)^{p+1}$ for every $y \geq 0$, 
	it follows that 
	\[
	(x+1)^{p}(y+1)^{p+1} - x^p y^{p+1} - (p+1) x^p y^p 
	\geq \left( y +1 \right)^{p+1} - \frac{1}{2} y^{p+1} - \frac{1}{2} y^p \geq 0
	\quad {\rm for\ any} \ y \geq 0.
	\] 
	On the other hand, if
	\[
	x^\eta \leq \frac{p-\eta}{p(p+1)},
	\]
	then from $C_{p,\eta} \geq 1$ we see that 
	\[
	(p-\eta) C_{p,\eta} - p(p+1) x^\eta \geq 0.
	\]
	Thus choosing 
	\[
	x_1 := \min \left\{ \left[ \frac{1}{2(p+1)} \right]^{1/p} , \ 
	\left[ \frac{p-\eta}{p(p+1)} \right]^{1/\eta} \right\} > 0,
	\]
	we obtain 
	\[
	\frac{\partial f}{\partial x} (x,y) \geq 0 \quad 
	{\rm for\ each} \ (x,y) \in [0,x_1] \times [0,\infty).
	\]
	By $f(0,y) = (y+1)^{p+1} - 1 - (p+1) y \geq 0 $ for all $y \geq 0$, 
	Step 1 holds. 
	
	\medskip
	
	\noindent
	\textbf{Step 2:} \textsl{There exists a $y_1>0$ such that $f(x,y) \geq 0$ 
		for all $(x,y) \in [x_1,\infty) \times [0,y_1]$. }
	
	\medskip
	
	Let $x \geq x_1$. 
	Since $C_{p,\eta} \geq 1$ and $(y+1)^{p+1} \geq 1 + (p+1) y$ holds for each $y \geq 0$, 
	we have 
	\[
	\begin{aligned}
	f(x,y) 
	&\geq (x+1)^{p+1} \left[ 1 + (p+1) y \right] 
	- x^{p+1} y^{p+1} - 1 - (p+1) ( x^p y^{p+1} + x^{p+1} y^p +y )  
	+  x_1^{1+\eta}
	\\
	&= \left[ (x + 1)^{p+1} - 1 - (p+1) x^{p+1} y^p \right] 
	+ \left[ x_1^{1+\eta} - (p+1) y \right]
	\\
	& \quad 
	+ y \left[ (p+1) (x+1)^{p+1} -  x^{p+1} y^p- (p+1) x^p y^p   \right].
	\end{aligned}
	\]
	Noting $x \geq x_1$, we may find a $c_{x_1} \in (0,1)$ so that 
	\[
	(x+1)^{p+1} - 1 \geq c_{x_1} (x+1)^{p+1} \quad {\rm for\ every} \ x \in [x_1,\infty). 
	\]
	If 
	\[
	y^p \leq \frac{c_{x_1}}{p+1},
	\]
	then we obtain 
	\[
	(x+1)^{p+1} - 1 - (p+1) x^{p+1} y^p 
	\geq c_{x_1} (x+1)^{p+1} - c_{x_1}x^{p+1} > 0 
	\quad {\rm for\ every} \ x \in [x_1,\infty).
	\]

	On the other hand, if 
	\[
	y^p \leq \frac{p+1}{p+2},
	\]
	then for every $x \in [x_1,\infty)$, 
	\[
	\begin{aligned}
	(p+1) (x+1)^{p+1} - x^{p+1} y^p - (p+1) x^p y^p
	\geq (x+1)^{p+1} \left[ p+1 - y^p - (p+1) y^p \right] \geq 0.
	\end{aligned}
	\]
	Hence, setting 
	\[
	y_1 := \min \left\{ \left(\frac{c_{x_1}}{p+1} \right)^{1/p} ,  \ 
	\left(\frac{p+1}{p+2} \right)^{1/p}, \ \frac{x_1^{1+\eta}}{p+1} \right\} > 0,
	\]
	we see that 
	\[
	f(x,y) \geq 0 \quad {\rm for\ each} \ (x,y) \in [x_1,\infty) \times [0,y_1]. 
	\]
	
	\medskip
	
	\noindent
	\textbf{Step 3:} \textsl{$f(x,y) \geq 0$ for all $(x,y) \in [x_1,\infty) \times [y_1,\infty)$. }
	
	\medskip

	Let $x \geq x_1$ and $y \geq y_1$. Note that 
	\[
	\begin{aligned}
	& x^{-p-1}f(x,y) 
	\\
	= & \,
	\left( 1 + \frac{1}{x} \right)^{p+1} (y+1)^{p+1} - y^{p+1} - \frac{1}{x^{p+1}} 
	- (p+1) \left( \frac{y^{p+1}}{x} + y^p + \frac{y}{x^{p+1}} \right) 
	+ C_{p,\eta} \left( \frac{y^{p+1}}{x^{1+\eta}} + x^{\eta - p} \right)
	.
	\end{aligned}
	\]
	Then we observe that for every $(x,y) \in [x_1,\infty) \times [y_1,\infty)$ 
	\[
	\begin{aligned}
	& \left( 1 + \frac{1}{x} \right)^{p+1} (y+1)^{p+1} 
	- y^{p+1} - (p+1) \left( \frac{y^{p+1}}{x} + y^p \right) 
	\\
	\geq \, & 
	\left( 1 + \frac{p+1}{x} \right) (y + 1)^{p+1} 
	- y^{p+1}  - (p+1) \frac{y^{p+1}}{x} - (p+1) y^p
	\\
	= \, & 
	\left[ (y+1)^{p+1} - y^{p+1} - (p+1) y^p \right] 
	+ \frac{p+1}{x} \left[ (y+1)^{p+1} - y^{p+1} \right]
	\\
	= \, & 
	y^{p+1} 
	\left[ \left( 1 + \frac{1}{y} \right)^{p+1} -1 - \frac{p+1}{y} \right] 
	+ \frac{p+1}{x} \left[ (y+1)^{p+1} - y^{p+1} \right]
	\\
	\geq \, & 0.
	\end{aligned}
	\]
	On the other hand, choosing 
	\[
	C_{p,\eta} \geq \max \left\{ 1, \ \frac{2}{ x_1^{p-\eta} y_1^{p+1} }, \ 
	\frac{2(p+1)}{x_1^{p-\eta} y_1^p }
	\right\},
	\]
	it follows from $0<\eta < p$ that for each $(x,y) \in [x_1,\infty) \times [y_1,\infty)$, 
	\[
	\begin{aligned}
	- \frac{1}{x^{p+1}} - \frac{(p+1)y}{x^{p+1}}+ C_{p,\eta} \frac{y^{p+1}}{x^{1+\eta}} 
	&= \frac{1}{x^{1+\eta}} 
	\left[ \left( \frac{C_{p,\eta} y^{p+1}}{2} - \frac{1}{x^{p-\eta}} \right) 
	+ y \left( \frac{C_{p,\eta} y^p}{2} - \frac{p+1}{x^{p-\eta}} \right) \right] 
	\\
	&\geq \frac{1}{x^{1+\eta}} 
	\left[ \left( \frac{C_{p,\eta} y_1^{p+1}}{2} - \frac{1}{x_1^{p-\eta}} \right) 
	+ y \left( \frac{C_{p,\eta} y_1^p}{2} - \frac{p+1}{x_1^{p-\eta}} \right) \right] \geq 0.
	\end{aligned}
	\]
	Hence, $f(x,y) \geq 0$ for each $(x,y) \in [x_1,\infty) \times [y_1,\infty)$ and 
	Step 3 holds.

	From Step 1 through Step 3, we complete the proof. 
\end{proof}

\subsection{Proof of Theorem \ref{theorem:1.1}}
\label{section:3.3}

	In this subsection, we shall prove Theorem \ref{theorem:1.1}. 
Let us take any minimizing sequence $(u_n) \subset M(\alpha) $ for $e(\alpha)$. 
By Lemmas \ref{lemma:2.2} and \ref{lemma:2.3}, without loss of generality, 
we may further suppose that
	\begin{equation}\label{eq:3.22}
		E(u_n) \to \ea, \quad 
		\left\| E'(u_n) + \lambda_{n,1} Q_1'(u_n) + \lambda_{n,2} Q_2'(u_n) \right\|_{H^\ast} \to 0, 
		\quad \left\| (u_{n,i})_- \right\|_{L^2} \to 0.
	\end{equation}
In addition, since $(u_n)$ and $(\lambda_{n,i})$ are bounded, one may assume that 
$u_n \rightharpoonup u_0$ weakly in $H$ and 
$\lambda_{n,i} \to \lambda_{0,i}$. Remark that 
	\[
		\gamma_{u_0,i} := \| u_{0,i} \|_{L^2}^2 =  2 Q_i(u_0) \leq \alpha_i.
	\]
Notice that if $\gamma_{u_0} = (\gamma_{u_0,1}, \gamma_{u_0,2} ) = (\alpha_1,\alpha_2)$, then 
$\| u_{n,i} - u_{0,i} \|_{L^2}^2 \to 0$ with $i=1,2$. 
Moreover, from \eqref{eq:3.22}, $\ea \leq E(u_0)$ and the weak lower semicontinuity, 
it follows that $\| \nabla u_{n,i} \|_{L^2} \to \| \nabla u_{0,i} \|_{L^2}$ and 
$\| u_n - u_0 \|_{H} \to 0$.

	Now, our aim is to prove $\gamma_{u_0} = (\alpha_1,\alpha_2)$. 
We argue by contradiction and suppose that 
	\[
		\gamma_{u_0} := (\gamma_{u_0,1}, \gamma_{u_0,2} ) \neq (\alpha_1,\alpha_2).
	\]
We first prove 

	\begin{lemma}\label{lemma:3.5}
		Let $Q^N := [0,1]^N$. Then  
		$\liminf_{n \to \infty} \sup_{z \in \Z^N} \| u_n - u_0 \|_{L^2(Q^N+z)} >0$. 
	\end{lemma}

\begin{proof}
Let us suppose $\sup_{z \in \Z^N}\| u_n - u_0 \|_{L^2(Q^N+z)} \to 0$. 
Then we infer that $u_n \to u_0$ strongly in $L^q(\RN)$ for every $2<q < 2^\ast$ 
(See \cite{W}, for instance). Therefore, we obtain 
	\begin{equation}\label{eq:3.23}
		e( \gamma_{u_0}) \leq  E(u_0) \leq \lim_{n\to \infty} E(u_n) = e(\alpha).
	\end{equation}
On the other hand, by Lemma \ref{lemma:2.1}(iii), it follows that 
	\[
		\ea  \leq e(\gamma_{u_0}) + e_\infty(\alpha - \gamma_{u_0}).
	\]
Since $\gamma_{u_0} \neq \alpha$, we have $e_{\infty}(\alpha - \gamma_{u_0}) < 0$, 
which contradicts \eqref{eq:3.23}. 
\end{proof}

Thanks to Lemma \ref{lemma:3.5}, we may find $(y_n) \subset \RN$ such that 
$\| u_n \|_{L^2(Q^N+y_n)} \to c_0>0$ and $|y_n| \to \infty$. Let 
	\[
		u_n(\cdot + y_n) \rightharpoonup w_0 = (w_{0,1},w_{0,2}) 
		\quad \text{weakly in} \ H.
	\]
Remark that 
	\begin{equation}\label{eq:3.24}
		w_0 \not \equiv (0,0).
	\end{equation}
Moreover, it follows from 
$|y_n| \to \infty$ as $n \to \infty$ that for $i = 1,2$, 
	\begin{equation}\label{eq:3.25}
		\begin{aligned}
		& \| u_{n,i} - u_{0,i} - w_{0,i} ( \cdot - y_n ) \|_{L^2}^2 
		\\
		= &\, \| u_{n,i} \|_{L^2}^2 + \|u_{0,i}\|_{L^2}^2 + \| w_{0,i} \|_{L^2}^2 
		- 2 \la u_{n,i} , u_{0,i} \ra_{L^2} - 2 \la u_{n,i}(\cdot + y_n) , w_{0,i} \ra_{L^2} + o(1)
		\\
		= &\, \| u_{n,i} \|_{L^2}^2 - \| u_{0,i} \|_{L^2}^2 - \| w_{0,i} \|_{L^2}^2 + o(1) .
		\end{aligned}
	\end{equation}
In particular, we have 
	\[
		\gamma_{w_0,i} := \| w_{0,i} \|_{L^2}^2 
		\leq \liminf_{n\to\infty} \left( \| u_{n,i} \|_{L^2}^2 - \| u_{0,i} \|_{L^2}^2 \right) 
		= \alpha_i - \gamma_{u_0,i}.
	\]

	Next, we show

	\begin{lemma}\label{lemma:3.6}
		We have $\| u_{n,i} - u_{0,i} - w_{0,i} (\cdot  - y_n) \|_{H^1} \to 0$ for $i=1,2$. 
		In addition, $\alpha = \gamma_{u_0} + \gamma_{w_0}$, 
		$E(u_0) = e(\gamma_{u_0})$, $E_\infty(w_0) = e_\infty(\gamma_{w_0})$ and 
		$\ea = e(\gamma_{u_0}) + e_\infty(\gamma_{w_0})$ 
		where $\gamma_{w_0} = ( \gamma_{w_0,1} , \gamma_{w_0,2} )$.
	\end{lemma}

\begin{proof}
We first claim that 
	\begin{equation}\label{eq:3.26}
		\| u_n - u_0 - w_0(\cdot - y_n) \|_{L^2} \to 0. 
	\end{equation}
We prove \eqref{eq:3.26} indirectly and suppose 
	\[
		\lim_{n\to \infty} \left\| u_n - u_0 - w_0(\cdot - y_n)  \right\|_{L^2} > 0.
	\]
By \eqref{eq:3.25}, one sees that 
$\gamma_{u_0} + \gamma_{w_0} \neq \alpha$. 
Applying the Brezis--Lieb Lemma and its proof (see, for instance, \cite[Lemma 1.32]{W}), 
for $2 \leq q \leq 2N/(N-2)_+$ and $ 1 \leq r \leq N/(N-2)_+$, we have 
	\[
		\begin{aligned}
			&\| u_{n,i} \|_{L^q}^q - \| u_{0,i} \|_{L^q}^q - \| u_{n,i} - u_{0,i} \|_{L^q}^q 
			= o(1), \\
			& \| (u_{n,i}-u_{0,i})(\cdot + y_n)) \|_{L^q}^q - \| w_{0,i} \|_{L^q}^q - 
			\| (u_{n,i}-u_{0,i})(\cdot + y_n)) - w_{0,i} \|_{L^q}^q = o(1),
			\\
			& \left\| |u_{n,i}|^{r} - |u_{0,i}|^{r} - | u_{n,i} - u_{0,i} |^{r} \right\|_{L^{2}} =o(1),
			\\
			&\left\| |(u_{n,i} - u_{0,i} ) (\cdot + y_n) |^{r} - |w_{0,i}|^{r} - | (u_{n,i} - u_{0,i})(\cdot + y_n) - w_{0,i}  |^{r} 
			\right\|_{L^{2}} =o(1).
		\end{aligned}
	\]
In particular, 
	\[
		| u_{n,i} (x) |^r = | u_{0,i} (x) |^r + 
		| w_{0,i} (x - y_n) |^r + 
		| u_{n,i} (x) - u_{0,i} (x) - w_{0,i} (x - y_{n}) |^r  + o(1) 
		\quad  \text{in} \ L^2(\RN).
	\]
Combining this with \eqref{eq:3.25} and $V_j(x) \to 0$ as $|x| \to \infty$, we obtain 
	\begin{equation}\label{eq:3.27}
		\begin{aligned}
			&E(u_n) - E(u_0) - E(w_0(\cdot - y_n)) - E(u_n-u_0-w_0(\cdot - y_n)) \to 0, \\ 
			&\delta_{0,i} := 
			\lim_{n\to\infty}\| u_{n,i}-u_{0,i}-w_{0,i}(\cdot - y_n) \|_{L^2}^2 
			= \alpha_i - \gamma_{u_0,i} - \gamma_{w_0,i}. 
		\end{aligned}
	\end{equation}
Here we notice that the assumption yields 
	\begin{equation}\label{eq:3.28}
		0 \leq \delta_{0,i},  \quad \delta_0:=(\delta_{0,1},\delta_{0,2}) \neq (0,0). 
	\end{equation}
Since 
$u_{n,i}(x)-u_{0,i}(x) - w_{0,i}(x-y_n) \to 0$ in $L^q_{\rm loc}(\RN)$ 
for $2 \leq q < 2N/(N-2)_+$ and 
	\[
		E(u_n-u_0-w_0(\cdot - y_n)) - E_\infty(u_n-u_0-w_0(\cdot - y_n)) 
		= \sum_{i=1}^2  \int_{\RN} \frac{V_{i}(x)}{2} |u_{n,i} (x) - u_{0,i}(x) - w_{0,i}(x-y_n) |^2 \rd x,
	\]
one observes that 
	\begin{equation}\label{eq:3.29}
		\int_{\RN} V_{i}(x) |u_{n,i} (x) - u_{0,i}(x) - w_{0,i}(x-y_n) |^2 \rd x 
		\to 0.
	\end{equation}
Moreover, from the fact that $E_\infty$ is uniformly continuous 
on each bounded set and $|y_n| \to \infty$, it follows that 
	\begin{equation}\label{eq:3.30}
		\liminf_{n\to\infty} E(u_n-u_0-w_0(\cdot - y_n)) \geq e_\infty(\delta_0), \quad 
		\liminf_{n\to\infty} E(w_0(\cdot - y_n)) \geq e_\infty(\gamma_{w_0}).
	\end{equation}
Thus \eqref{eq:3.27}, \eqref{eq:3.29} and \eqref{eq:3.30} yield
	\begin{equation}\label{eq:3.31}
		e(\alpha) \geq e(\gamma_{u_0}) + e_\infty(\gamma_{w_0}) + 
		e_\infty(\delta_0). 
	\end{equation}

	By noting that for $e_\infty(\alpha)$, the compactness of all minimizing sequences up to 
translations is equivalent to the strict subadditivity condition
	\[
		e_\infty (\alpha) < e_\infty(\gamma) + e_\infty(\alpha - \gamma) 
		\quad {\rm for\ all} \ \gamma \in [0,\alpha_1] \times [0,\alpha_2] \ 
		{\rm with} \ \gamma \neq (0,0), \alpha.
	\]
Therefore, by the result of \cite{GJ-16} with \eqref{eq:3.24} and \eqref{eq:3.28}, we have 
	\[
		e_\infty( \gamma_{w_0})  + e_\infty(\delta_0) 
		> e_\infty(\delta_0 + \gamma_{w_0})
		= e_\infty(\alpha-\gamma_{u_0}).
	\]
It follows from \eqref{eq:3.31} that $e(\alpha) > e(\gamma_{u_0}) + e_\infty(\alpha - \gamma_{u_0})$, 
however this contradicts Lemma \ref{lemma:2.1}(iii) and \eqref{eq:3.26} holds.

	By \eqref{eq:3.26}, one sees $\| u_{n,i} - u_{0,i} - w_{0,i}(\cdot - y_n) \|_{L^q} \to 0$ 
for each $2 \leq q < 2N/(N-2)_+$. Note also that \eqref{eq:3.29} holds in this case. 
Therefore, as in the proof of Step 2, we get 
	\begin{equation}\label{eq:3.32}
		\begin{aligned}
			\ea &=  E(u_n) + o(1)
			\\
			&=  E(u_0) + E(w_0(\cdot - y_n)) +  \frac{1}{2} \sum_{i=1}^{2} 
			\left\|  \nabla u_{n,i} - \nabla u_{0,i} - \nabla w_{0,i}(\cdot - y_n) \right\|_{L^2}^2
			\\
			& \quad
			+  \frac{1}{2} \sum_{i=1}^2
			\int_{\RN} V_i(x) | u_{n,i}(x) - u_{0,i}(x) - w_{0,i} (x-y_n) |^2 \rd x
			+ o(1)
			\\
			&\geq  e(\gamma_{u_0}) + e_\infty ( \gamma_{w_0} ) 
			+ \frac{1}{2} \sum_{i=1}^2 
						\| \nabla u_{n,i} - \nabla u_{0,i} - \nabla w_{0,i}(\cdot - y_n) \|_{L^2}^2
			+ o(1) 
		\end{aligned}
	\end{equation}
which implies 
	\[
		\ea \geq e(\gamma_{u_0}) + e_\infty(\gamma_{w_0}). 
	\]
Moreover, it follows from \eqref{eq:3.32} and Lemma \ref{lemma:2.1}(iii) that 
	\[
		\begin{aligned}
			& e(\alpha) = e (\gamma_{u_0}) + e ( \gamma_{w_0}), \quad 
			\| \nabla u_{n,i} - \nabla u_{0,i} - \nabla w_{0,i}(\cdot - y_n) \|_{L^2} \to 0 
			\ (i = 1,2),
			\\
			&
			E(u_0) = e(\gamma_{u_0}), \quad 
			E_\infty(w_0) = e_\infty ( \gamma_{w_0} ) .
		\end{aligned}
	\]
Hence, by combining this with \eqref{eq:3.26}, Lemma \ref{lemma:3.6} holds. 
\end{proof}

	To proceed, we recall \eqref{eq:3.22}, $\lambda_{n,i} \to \lambda_{0,i}$, and 
$u_n \rightharpoonup u_0$ and $u_n(\cdot + y_n) \rightharpoonup w_0$ 
weakly in $H$. 
Hence, $u_0$ (resp. $w_0$) is a nonnegative solution of \eqref{eq:3.1} 
(resp. \eqref{eq:3.3}). 
Our next task is to show $0< \lambda_{0,i}$ for $i=1,2$:

	\begin{lemma}\label{lemma:3.7}
		Suppose \textup{(p2)} in addition to \textup{(V1)} and \textup{(p1)}. 
		Then $\lambda_{0,i} > 0$ $(i=1,2)$ hold. 
	\end{lemma}

\begin{proof}
We remark that by Lemma \ref{lemma:3.6} and $\alpha_i > 0$ ($i=1,2)$, 
for each $i=1,2$, 
either $u_{0,i} \not\equiv 0$ or $w_{0,i} \not\equiv 0$. 
In what follows, we only show that $\lambda_{0,1} > 0$ holds when $u_{0,1} \not\equiv 0$, 
but other cases can be treated similarly.

	Assume $u_{0,1} \not \equiv 0$. Then $u_{0,1} > 0$ in $\RN$ 
by the strong maximum principle. 
If $\lambda_{0,1} \leq 0$, then by $u_{0,i} \geq 0$, $V_i(x) \leq 0$ and $\beta >0$, 
it follows from \eqref{eq:3.1} that 
	\[
		-\Delta u_{0,1} \geq \mu_1 u_{0,1}^{2p_1 + 1} \quad {\rm in} \ \RN.
	\]
However, this inequality does not have any positive solution for 
$2p_1 +1 \leq N/(N-2)$ when $N \geq 3$. 
See, for instance, \cite[Theorem 8.4]{QS-07}. 
When $N=1$, it is easy to see that 
the inequality admits no positive solution. If $N=2$, since $u_{0,1}$ is bounded 
by elliptic regularity and $u_{0,1}(x) \to 0$ as $|x| \to \infty$, 
\cite[Theorem 29]{PW-84} yields $u_{0,i} \equiv 0$. 
Therefore, we have a contradiction and $\lambda_{0,1} > 0$ holds 
provided $u_{0,1}\not \equiv 0$. 
\end{proof}

	From Lemma \ref{lemma:3.7}, 
without loss of generality, we may assume \eqref{eq:3.2}: $0 < \lambda_{0,1} \leq \lambda_{0,2}$. 
Thus, we note that Lemma \ref{lemma:3.1} is applicable to $u_0$ and $w_0$.

	Now under (V1) and (p1)--(p2), 
we derive a contradiction and complete the proof of Theorem \ref{theorem:1.1}:

	\begin{proof}[Proof of Theorem \ref{theorem:1.1}]
By Lemma \ref{lemma:3.6}, we have 
	\begin{equation}\label{eq:3.33}
		\begin{aligned}
			&\ea = e(\gamma_{u_0}) + e_\infty(\gamma_{w_0}), \quad 
			E(u_0) = e(\gamma_{u_0}), \quad 
			E_\infty(w_0) = e_\infty (\gamma_{w_0}), \quad 
			w_0 \neq (0,0),
			\\
			&\text{$u_0$ (resp. $w_0$) is a nonnegative solution of \eqref{eq:3.1} (resp. \eqref{eq:3.3}) 
			with \eqref{eq:3.2}}. 
		\end{aligned}
	\end{equation}
In what follows, we will find a $u \in M(\alpha)$ such that 
	\begin{equation}\label{eq:3.34}
		E(u) < e(\gamma_{u_0}) + e_{\infty} (\gamma_{w_0}),
	\end{equation}
which leads a contradiction and we conclude that $(u_n)$ has a strongly convergent 
subsequence in $H$.

	To show \eqref{eq:3.34}, set 
	\[
		w_{n,i}(x) := w_{0,i}(x - n \mathbf{e}_1 ), \quad 
		\tau_{n,i} := \frac{\sqrt{\alpha_i}}{\| u_{0,i} + w_{n,i} \|_{L^2} }, \quad 
		\kappa_{n,i} := \la u_{0,i} , w_{n,i} \ra_{L^2}. 
	\]
Remark that 
	\[
		\left( \tau_{n,1} [ u_{0,1} + w_{n,1} ], \ 
		\tau_{n,2} [ u_{0,2} + w_{n,2} ] \right) \in M(\alpha), \quad 
		0 \leq \kappa_{n,i} \to 0. 
	\]
Our goal is to prove 
	\begin{equation}\label{eq:3.35}
	E( \tau_{n,1} [ u_{0,1} + w_{n,1} ], \ 
	\tau_{n,2} [ u_{0,2} + w_{n,2} ] ) 
	< e(\gamma_{u_0}) + e_{\infty} (\gamma_{w_0})
	\end{equation}
	for sufficiently large $n$. 
We divide our proof into several steps. 

\medskip

\noindent
{\bf Step 1:} {\sl In the following cases, \eqref{eq:3.35} holds: 
	(i) $u_0 = (0,0)$, (ii) $u_0=(u_{0,1},0)$ and $w_0=(0,w_{0,2})$, 
	(iii) $u_0=(0,u_{0,2})$ and $w_0=(w_{0,1},0)$. }

\medskip

	In each case, we have $\kappa_{n,1} = \kappa_{n,2} = 0$ and $\tau_{n,1} = 1 = \tau_{n,2}$. 
When $u_0=(0,0)$, notice that $\ea = e_\infty(\gamma_{w_0})$ 
thanks to \eqref{eq:3.33} and that 
	\[
		\begin{aligned}
			E(w_{n,1},w_{n,2}) 
			&= E_\infty(w_{n,1}, w_{n,2}) 
			+ \frac{1}{2} \int_{\RN} V_1(x) w_{n,1}^2 + V_2(x) w_{n,2}^2 \rd x
			\\
			&= \ea + \frac{1}{2} \int_{\RN} V_1(x) w_{n,1}^2 + V_2(x) w_{n,2}^2 \rd x.
		\end{aligned}
	\]
Since $V_j(x) \leq 0$ ($j=1,2$), $w_j>0$ ($j=1,2$) and 
either $V_1(x) \not\equiv 0$ or $V_2(x) \not \equiv 0$, 
we get $E(w_{n,1},w_{n,2}) < \ea $ and \eqref{eq:3.35} holds.

	In the case $u_0=(u_{0,1},0)$ and $w_0 = (0,w_{0,2})$, 
it follows from $u_{0,1},w_{0,2} > 0$ and $\beta > 0$ that 
	\[
		E( u_{0,1} , w_{n,2} ) 
		\leq E(u_{0,1}, 0) + E_\infty(0,w_{n,2}) 
		- \frac{\beta}{p_3+1} \int_{\RN} u_{0,1}^{p_3+1} w_{n,2}^{p_3+1} \rd x 
		< e(\gamma_{u_0}) + e_\infty ( \gamma_{w_0} ) = \ea.
	\]
Hence, \eqref{eq:3.35} holds. 
The case $u_0=(0,u_{0,2})$ and $w_0=(w_{0,1},0)$ can be treated similarly. 
Thus Step 1 holds.

\medskip

From Step 1, we may suppose either $u_{0,1}, w_{0,1} > 0$ or 
$u_{0,2}, w_{0,2} > 0$, hence we may assume that 
	\begin{equation}\label{eq:3.36}
		\text{either $\kappa_{n,1} > 0$ or $0=\kappa_{n,1} < \kappa_{n,2}$}. 
	\end{equation}

\medskip

\noindent
{\bf Step 2:} 
{\sl (i) If $\kappa_{n,1} > 0$, then for each $\lambda < \lambda_{0,1} $, 
there exist $c>0$, which is independent of $\lambda$, and 
$0 < C_{\lambda}$ such that 
	\begin{equation}\label{eq:3.37}
		c \left( 1 + n \right)^{-(N-1)/2} \exp \left( - \sqrt{\lambda_{0,1}} n \right)
		 \leq \kappa_{n,1} 
		\leq C_{\lambda} \exp \left( - \sqrt{\lambda} n \right) 
		\quad {\rm for\ all} \ n \geq 1. 
	\end{equation}
Moreover, if $\lambda_{0,1} = \lambda_{0,2}$ and $\kappa_{n,2} > 0$, 
then $\kappa_{n,2}$ also satisfies \eqref{eq:3.37}. 
If $\lambda_{0,1} < \lambda_{0,2}$, then 
there exist $\theta_1 > 1$ and $C_0>0$ such that 
	\begin{equation}\label{eq:3.38}
		\kappa_{n,2} \leq C_0 \kappa_{n,1}^{\theta_1} \quad 
		{\rm for\ all} \ n \geq 1.
	\end{equation}

\noindent
(ii) Suppose $\kappa_{n,1}=0$ and $\kappa_{n,2}>0$. 
When either $p_3 \geq 1$ or else 
$0 < p_3 < 1$ and $ \lambda_{0,3} = (1+p_3)^2 \lambda_{0,1} / (1-p_3)^2 > \lambda_{0,2}$, 
for every $\lambda < \lambda_{0,2}$, 
there exist $c > 0$, which is independent of $\lambda$, and 
$0< C_{\lambda}$ such that 
	\begin{equation}\label{eq:3.39}
		\begin{aligned}
			c (1+n)^{-(N-1)/2} \exp \left( - \sqrt{\lambda_{0,2}} n \right)
			\leq \kappa_{n,2} \leq 
			C_{\lambda} \exp \left( - \sqrt{\lambda} n \right)
			\quad {\rm for\ all} \ n \geq 1. 
		\end{aligned}
	\end{equation}
When $0 < p_3 < 1$ and $ \lambda_{0,3}  \leq \lambda_{0,2}$, 
for every $\lambda_1 < \lambda_{0,3} < \lambda_2$ 
there exist $C_{\lambda_1}, C_{\lambda_2}>0$ such that 
	\begin{equation}\label{eq:3.40}
		\begin{aligned}
			C_{\lambda_2} \exp \left( - \sqrt{\lambda_2} n \right) 
			\leq \kappa_{n,2} \leq 
			C_{\lambda_1} \exp \left( - \sqrt{\lambda_1} n \right)
			\quad {\rm for\ all} \ n \geq 1. 
		\end{aligned}
	\end{equation}
}

\medskip

(i) By $\kappa_{n,1} > 0$, one has $u_{0,1}, w_{0,1} > 0$. 
From Lemma \ref{lemma:3.1} (i), for any $\lambda < \lambda_{0,1}$, 
it follows that 
	\[
		\begin{aligned}
			&u_{0,1}(x)  \exp \left( \sqrt{\lambda} |x| \right) 
			+ w_{0,1} (x) \exp \left( \sqrt{\frac{\lambda+\lambda_{0,1}}{2}} |x| \right)
			\to 0 \quad {\rm as} \ |x| \to \infty,
			\\
			&
			c (1+|x|)^{-(N-1)/2} \exp \left( - \sqrt{\lambda_{0,1}} |x| \right) 
			\leq u_{0,1} (x),
			\quad 
			c \exp \left( - \sqrt{\lambda_{0,1} + 1} |x| \right) \leq w_{0,1}(x).
		\end{aligned}
	\]
Thus, Lemma \ref{lemma:3.3} asserts 
	\[
		\exp \left( \sqrt{\lambda} n \right) \kappa_{n,1} \to 0
	\]
and 
	\[
		\begin{aligned}
			& (1+n)^{(N-1)/2} \exp \left( \sqrt{\lambda_{0,1}} n \right) \kappa_{n,1} 
			\\
			\geq & \, 
			c^2 (1+n)^{(N-1)/2} \exp \left( \sqrt{\lambda_{0,1}} n \right) 
			\\
			& \quad \times 
			\int_{\RN} 
			(1+|y-n\mathbf{e}_1|)^{-(N-1)/2} \exp \left( - \sqrt{\lambda_{0,1}} |y-n\mathbf{e}_1| \right) 
			\exp \left( -\sqrt{\lambda_{0,1} + 1} |y| \right)
			\rd y
			\\
			\to & \,  c^2 
			\int_{\RN} \exp \left( - \sqrt{\lambda_{0,1} + 1} |y| \right) 
			\exp \left( \sqrt{\lambda_{0,1}} \mathbf{e}_1 \cdot y \right) \rd y > 0.
		\end{aligned}
	\]
Hence, \eqref{eq:3.37} holds.

	If $\lambda_{0,1} = \lambda_{0,2}$ and $\kappa_{n,2} > 0$ hold, 
then from Lemma \ref{lemma:3.1} (ii) and (iii), 
we observe that $u_{0,2}$ satisfies the decay estimate in Lemma \ref{lemma:3.1} (i). 
Thus, we can argue as in the above and see that $\kappa_{n,2}$ satisfies \eqref{eq:3.37}.

	Finally, if $\lambda_{0,1} < \lambda_{0,2}$, by Lemma \ref{lemma:3.1} (ii) and (iii), 
one can find an $\delta_0>0$ such that 
	\[
		0 \leq u_{0,2}(x) \leq C \exp \left( -\sqrt{\lambda_{0,1} + 2\delta_0} |x| \right), \quad 
		w_{0,2}(x) \leq C \exp \left( -\sqrt{\lambda_{0,1} + 3 \delta_0} |x| \right). 
	\]
Hence, Lemma \ref{lemma:3.3} yields 
	\[
		\exp \left( \sqrt{\lambda_{0,1} + \delta_0} n \right) \kappa_{n,2} \to 0.
	\]
Thus, \eqref{eq:3.38} holds for some $\theta_1>1$ and $C_0>0$ 
due to \eqref{eq:3.37}.

	(ii) By $\kappa_{n,1} = 0$, remark that either 
$u_0=(0,u_{0,2})$ or $w_0 = (0,w_{0,2})$. 
Since the arguments for both cases are similar, 
we only deal with the case $u_0=(0,u_{0,2})$. 
From $\kappa_{n,2} > 0$ and $\alpha = \gamma_{u_0} + \gamma_{w_0}$, 
we have $w_0 = (w_{0,1},w_{0,2})$ with 
$w_{0,j} > 0$ ($j=1,2$). Moreover, by Lemma \ref{lemma:3.1} (ii), 
for every $\lambda \in (0,\lambda_{0,2})$ there exists a $C_\lambda > 0$ 
such that 
	\begin{equation}\label{eq:3.41}
		c_0 \left( 1 + |x| \right)^{-(N-1)/2} 
		\exp \left( - \sqrt{\lambda_{0,2}} |x| \right)
		\leq u_{0,2}(x) 
		\leq C_{\lambda} \exp \left( - \sqrt{\lambda} |x| \right). 
	\end{equation}

	When either $p_3 \geq 1$ or else 
$0 < p_3 < 1$ and $ \lambda_{0,3} = (1+p_3)^2 \lambda_{0,1} / (1-p_3)^2 > \lambda_{0,2}$, 
by Lemma \ref{lemma:3.1} (iii), $w_{0,2}$ satisfies 
the same estimate to \eqref{eq:3.41}:
	\[
		c_0 \left( 1 + |x| \right)^{-(N-1)/2} 
		\exp \left( - \sqrt{\lambda_{0,2}} |x| \right)
		\leq w_{0,2}(x) 
		\leq C_{\lambda} \exp \left( - \sqrt{\lambda} |x| \right)
	\]
for every $\lambda \in (0,\lambda_{0,2})$. 
Hence, as in (i), we can prove \eqref{eq:3.39}.

	Next, suppose $0<p_3<1$ and 
$\lambda_{0,3} \leq \lambda_{0,2}$. 
Lemma \ref{lemma:3.1} (iii) gives 
	\[
		C_{\lambda_2} \exp \left( - \sqrt{\lambda_2} |x| \right) 
		\leq w_{0,2}  (x) \leq 
		C_{\lambda_1} \exp \left( - \sqrt{\lambda_1} |x| \right)
	\]
for every $\lambda_1 \in(0, \lambda_{0,3})$ and $ \lambda_2 \in (\lambda_{0,3},\infty) $.  
The rest of arguments is similar to the above and 
we can prove \eqref{eq:3.40}. 

\medskip

\noindent
{\bf Step 3: } 
{\sl Estimate of $E( \tau_{n,1} [ u_{0,1} + w_{n,1} ], 
\tau_{n,2} [ u_{0,2} + w_{n,2}  ] ) $}.

\medskip

Since $\alpha_i = \| u_{0,i} \|_{L^2}^2 + \| w_{0,i} \|_{L^2}^2$, 
we have 
	\[
		\| u_{0,i} + w_{n,i} \|_{L^2}^2 = \alpha_i + 2 \kappa_{n,i}.
	\]
Hence, 
	\begin{equation}\label{eq:3.42}
		\tau_{n,i}^2 = \frac{\alpha_i}{\alpha_i + 2 \kappa_{n,i}} 
		= \frac{1}{1+ 2 \kappa_{n,i}/\alpha_i} 
		= 1 - \frac{2\kappa_{n,i}}{\alpha_i} + O(\kappa_{n,i}^2), \quad 
		\tau_{n,i} = 1 - \frac{\kappa_{n,i}}{\alpha_i} + O(\kappa_{n,i}^2)
	\end{equation}
and $\tau_{n,i} = 1$ if $\kappa_{n,i} = 0$.

	Recalling the notation in \eqref{eq:2.1} and noting 
	\begin{equation}\label{eq:3.43}
		\begin{aligned}
			& E( \tau_{n,1} [ u_{0,1} + w_{n,1} ], \ \tau_{n,2} [ u_{0,2} + w_{n,2} ] )
			\\
			= &\, 
			\sum_{i=1}^2 \frac{\tau_{n,i}^2}{2} 
			\left\| u_{0,i} + w_{n,i} \right\|_{V_i}^2 
			- \int_{\RN} \sum_{i=1}^2 \frac{\mu_i \tau_{n,i}^{2p_i+2}}{2p_i+2} 
			(u_{0,i} + w_{n,i})^{2p_i+2} \rd x 
			\\
			& \quad 
			- \int_{\RN} \frac{\beta \tau_{n,1}^{p_3+1} \tau_{n,2}^{p_3+1} }{p_3+1} 
			(u_{0,1} + w_{n,1})^{p_3+1} (u_{0,2}+w_{n,2})^{p_3+1} \rd x,
		\end{aligned}
	\end{equation}
we compute each term.

	First, it follows from \eqref{eq:3.42} and $\la u_{0,i} , w_{n,i} \ra_{V_i} \to 0$ 
as $n \to \infty$ that 
	\begin{equation}\label{eq:3.44}
		\begin{aligned}
			&\frac{\tau_{n,i}^2}{2} \| u_{0,i} + w_{n,i} \|_{V_i}^2 
			\\
			= &\, \frac{1}{2} \left( 1 - \frac{2\kappa_{n,i}}{\alpha_i} + O(\kappa_{n,i}^2) \right) 
			\left( \| u_{0,i} \|_{V_i}^2 + \| w_{n,i} \|_{V_i}^2 + 2 \la u_{0,i} , w_{n,i} \ra_{V_i} \right)
			\\
			= &\, \frac{1}{2} 
			\left( \| u_{0,i} \|_{V_i}^2 + \| w_{n,i} \|_{V_i}^2 + 2 \la u_{0,i} , w_{n,i} \ra_{V_i} \right)
			- \frac{\kappa_{n,i}}{\alpha_i} 
			\left( \| u_{0,i} \|_{V_i}^2 + \| w_{n,i} \|_{V_i}^2 \right)
			\\
			& \, 
			+ O(\kappa_{n,i}^2) + O \left( \kappa_{n,i} \la u_{0,i}, w_{n,i} \ra_{V_i} \right)
			\\
			= & \, \frac{1}{2} 
			\left( \| u_{0,i} \|_{V_i}^2 + \| \nabla w_{0,i} \|_{L^2}^2 
			+ 2 \la u_{0,i} , w_{n,i} \ra_{V_i} \right)
			- \frac{\kappa_{n,i}}{\alpha_i} 
			\left( \| u_{0,i} \|_{V_i}^2 + \| \nabla w_{0,i} \|_{L^2}^2 \right)
			\\
			& \,  
			+ \left( \frac{1}{2} - \frac{\kappa_{n,i}}{\alpha_i} \right) \int_{\RN} V_i(x) w_{n,i}^2 \rd x
			+ O(\kappa_{n,i}^2) + O \left( \kappa_{n,i} \la u_{0,i}, w_{n,i} \ra_{V_i} \right).
		\end{aligned}
	\end{equation}
Since $u_0$ is a solution of \eqref{eq:3.1}, one sees 
	\begin{equation}\label{eq:3.45}
		\begin{aligned}
			\la u_{0,1} , w_{n,1} \ra_{V_1} 
			&= - \lambda_{0,1} \kappa_{n,1} 
			+ \int_{\RN} \mu_1 u_{0,1}^{2p_1+1} w_{n,1} + 
			\beta u_{0,1}^{p_3} u_{0,2}^{p_3+1} w_{n,1} \rd x,
			\\
			\la u_{0,2} , w_{n,2} \ra_{V_2} 
			&= - \lambda_{0,2} \kappa_{n,2} 
			+ \int_{\RN} \mu_2 u_{0,2}^{2p_2+1} w_{n,2} + 
			\beta u_{0,1}^{p_3+1} u_{0,2}^{p_3} w_{n,2} \rd x.
		\end{aligned}
	\end{equation}
By Lemma \ref{lemma:3.1}, we notice that 
	\[
		u_{0,1}^{2p_1+1}(x) + u_{0,1}^{p_3}(x) u_{0,2}^{p_3+1} (x)
		+ \mu_2 u_{0,2}^{2p_2+1} (x) + \beta u_{0,1}^{p_3+1} (x)
		u_{0,2}^{p_3} (x)
		\leq C \exp \left( - \sqrt{\lambda_{0,1}} |x| \right).
	\]
From Lemma \ref{lemma:3.3} it follows that 
	\[
		\int_{\RN} \mu_1 u_{0,1}^{2p_1+1} w_{n,1} + 
		\beta u_{0,1}^{p_3} u_{0,2}^{p_3+1} w_{n,1} 
		+ \mu_2 u_{0,2}^{2p_2+1} w_{n,2} + 
		\beta u_{0,1}^{p_3+1} u_{0,2}^{p_3} w_{n,2} \rd x 
		\leq C_0 \exp \left( - \sqrt{\frac{\lambda_{0,1}}{2}} n  \right).
	\]
Noting Step 2, we may find a $\theta_2 \in (1,2)$ such that 
	\begin{equation}\label{eq:3.46}
		\kappa_{n,i} \la u_{0,i} , w_{n,i} \ra_{V_i} = o( \kappa_{n,i}^{\theta_2} ) 
		\quad {\rm for} \ i=1,2. 
	\end{equation}
Since $\kappa_{n,i} \to 0$, 
we may assume $1/2 - \kappa_{n,i} / \alpha_i \geq 0$ and 
it follows from $V_i(x) \leq 0$,  \eqref{eq:3.44}--\eqref{eq:3.46} and $1 < \theta_2 < 2$ that 
	\begin{equation}\label{eq:3.47}
		\begin{aligned}
			&\frac{\tau_{n,1}^2}{2} \| u_{0,1} + w_{n,1} \|_{V_1}^2 
			\\
			\leq & \, 
			\frac{1}{2} \left( \| u_{0,1} \|_{V_1}^2 + \| \nabla w_{0,1} \|_{L^2}^2 \right) 
			- \lambda_{0,1} \kappa_{n,1} + 
			\int_{\RN} \mu_1 u_{0,1}^{2p_1+1} w_{n,1} + 
			\beta u_{0,1}^{p_3} u_{0,2}^{p_3+1} w_{n,1} \rd x
			\\
			&
			- \frac{\kappa_{n,1}}{\alpha_1} 
			\left( \| u_{0,1} \|_{V_1}^2 + \| \nabla w_{0,1} \|_{L^2}^2 \right) 
			+ o(\kappa_{n,1}^{\theta_2}) 
		\end{aligned}
	\end{equation}
and 
	\begin{equation}\label{eq:3.48}
		\begin{aligned}
			&\frac{\tau_{n,2}^2}{2} \| u_{0,2} + w_{n,2} \|_{V_2}^2 
			\\
			\leq & \, 
			\frac{1}{2} \left( \| u_{0,2} \|_{V_2}^2 + \| \nabla w_{0,2} \|_{L^2}^2 \right) 
			- \lambda_{0,2} \kappa_{n,2} + 
			\int_{\RN} \mu_2 u_{0,2}^{2p_2+1} w_{n,2} + 
			\beta u_{0,1}^{p_3+1} u_{0,2}^{p_3} w_{n,2} \rd x
			\\
			&
			- \frac{\kappa_{n,2}}{\alpha_2} 
			\left( \| u_{0,2} \|_{V_2}^2 + \| \nabla w_{0,2} \|_{L^2}^2 \right) 
			+ o(\kappa_{n,2}^{\theta_2}) .
		\end{aligned}
	\end{equation}

	Next, it follows from \eqref{eq:3.42} and Lemma \ref{lemma:3.4} that 
	\[
		\begin{aligned}
			& - \frac{\mu_1 \tau_{n,1}^{2p_1+2} }{2p_1+2} 
			\int_{\RN} ( u_{0,1} + w_{n,1} )^{2p_1+2} \rd x 
			\\
			\leq & \, 
			 - \frac{\mu_1  }{2p_1+2} \left( 1 - \frac{2p_1+2}{\alpha_1} \kappa_{n,1} \right)
			 \\
			 & \, \times 
			 \int_{\RN} u_{0,1}^{2p_1+2} + w_{n,1}^{2p_1+2} 
			 + (2p_1+2) \left( u_{0,1}^{2p_1+1} w_{n,1} + u_{0,1} w_{n,1}^{2p_1+1}  \right) 
			  - C_{p_1} u_{0,1}^{p_1+1} w_{n,1}^{p_1+1} \rd x + O(\kappa_{n,1}^2).
		\end{aligned}
	\]
Thanks to $1 < p_1 + 1$ and Lemma \ref{lemma:3.1}, 
we may find an $\eta_1>0$ such that 
	\[
		u_{0,1}^{p_1+1} (x) + w_{0,1}^{p_1+1} (x) \leq 
		C \exp  \left( - \sqrt{\lambda_{0,1} + \eta_1} |x|  \right).
	\]
Therefore, exploiting Lemma \ref{lemma:3.3}, 
we can find a $\theta_3 = \theta_3 (p_1) \in (1,2)$ such that 
	\[
		\kappa_{n,1} \int_{\RN} u_{0,1}^{2p_1+1} w_{n,1} + u_{0,1} w_{n,1}^{2p_1+1} \rd x
		+ 
		\int_{\RN} u_{0,1}^{p_1+1} w_{n,1}^{p_1+1} \rd x 
		= o( \kappa_{n,1}^{\theta_3} ).
	\]
Hence, 
	\begin{equation}\label{eq:3.49}
		\begin{aligned}
			& - \frac{\mu_1 \tau_{n,1}^{2p_1+2} }{2p_1+2} 
			\int_{\RN} ( u_{0,1} + w_{n,1} )^{2p_1+2} \rd x 
			\\
			\leq & \, 
			- \frac{\mu_1}{2p_1+2} \int_{\RN} u_{0,1}^{2p_1+2} + w_{0,1}^{2p_1+2} \rd x 
			- \mu_1 \int_{\RN} 
			u_{0,1}^{2p_1+1} w_{n,1} + u_{0,1} w_{n,1}^{2p_1+1} 
			\rd x
			 \\
			 & \, 
			 + \frac{\kappa_{n,1}}{\alpha_1} \mu_1 
			 \int_{\RN} u_{0,1}^{2p_1+2} + w_{0,1}^{2p_1 + 2} \rd x 
			 + o(\kappa_{n,1}^{\theta_3})
			 .
		\end{aligned}
	\end{equation}
Similarly, we have 
	\begin{equation}\label{eq:3.50}
		\begin{aligned}
			& - \frac{\mu_2 \tau_{n,2}^{2p_2+2} }{2p_2+2} 
			\int_{\RN} ( u_{0,2} + w_{n,2} )^{2p_2+2} \rd x 
			\\
			\leq & \, 
			- \frac{\mu_2}{2p_2+2} \int_{\RN} u_{0,2}^{2p_2+2} + w_{0,2}^{2p_2+2} \rd x 
			- \mu_2 \int_{\RN} 
			u_{0,2}^{2p_2+1} w_{n,2} + u_{0,2} w_{n,2}^{2p_2+1} 
			\rd x
			 \\
			 & \, 
			 + \frac{\kappa_{n,2}}{\alpha_2} \mu_2 
			 \int_{\RN} u_{0,2}^{2p_2+2} + w_{0,2}^{2p_2 + 2} \rd x 
			 + o(\kappa_{n,2}^{\theta_4})
		\end{aligned}
	\end{equation}
for some $\theta_4 = \theta_4(p_2) \in (1,2)$.

	On the other hand, choosing $\eta = p_3/2$ in Lemma \ref{lemma:3.4} (ii), 
it follows from \eqref{eq:3.42} that 
	\[
		\begin{aligned}
			& - \frac{\beta \tau_{n,1}^{p_3+1} \tau_{n,2}^{p_3+1} }{p_3+1} 
			\int_{\RN} ( u_{0,1} + w_{n,1} )^{p_3+1} 
			 ( u_{0,2} + w_{n,2} )^{p_3+1} \rd x 
			\\
			\leq & \, 
			- \frac{\beta}{p_3+1} \left( 1 - \frac{p_3+1}{\alpha_1} \kappa_{n,1} \right) 
			\left( 1 - \frac{p_3+1}{\alpha_2} \kappa_{n,2} \right)
			 \\
			 & \, \times 
			 \int_{\RN} 
			 u_{0,1}^{p_3+1} u_{0,2}^{p_3+1} + w_{0,1}^{p_3+1} w_{0,2}^{p_3+1} 
			 +(p_3+1) \left( u_{0,1}^{p_3} u_{0,2}^{p_3+1} w_{n,1} 
			 + u_{0,1}^{p_3+1} u_{0,2}^{p_3} w_{n,2} 
			 + u_{0,2} w_{n,1}^{p_3+1} w_{n,2}^{p_3}
			  \right) \rd x 
			 \\
			 & \, + 
			 \frac{\beta}{p_3+1} \left( 1 - \frac{p_3+1}{\alpha_1} \kappa_{n,1} \right)
			\left( 1 - \frac{p_3+1}{\alpha_2} \kappa_{n,2} \right) C_{p_3} 
			 \int_{\RN} 
			 u_{0,1}^{p_3/2} u_{0,2}^{p_3+1} w_{n,1}^{1+p_3/2} 
			 + u_{0,1}^{1+p_3/2} w_{n,1}^{p_3/2} w_{n,2}^{p_3+1}
			 \rd x 
			 \\
			 & \, 
			 + O(\kappa_{n,1}^2 + \kappa_{n,2}^2).
		\end{aligned}
	\]
As in the above, if $\kappa_{n,1} >0$, then by Lemma \ref{lemma:3.1} and Step 2, 
we may find a $\theta_5 \in (1,2)$ such that 
	\[
		\begin{aligned}
			o ( \kappa_{n,1}^{\theta_5}  )
			&= \sum_{i=1}^{2} \kappa_{n,i} \int_{\RN} 
			u_{0,1}^{p_3} u_{0,2}^{p_3+1} w_{n,1} + u_{0,1}^{p_3+1} u_{0,2}^{p_3} w_{n,2} 
			+ u_{0,2} w_{n,1}^{p_3+1} w_{n,2}^{p_3} \rd x
			\\
			& \quad + 
			 \int_{\RN} u_{0,1}^{p_3/2} u_{0,2}^{p_3+1} w_{n,1}^{1+p_3/2} 
			  + u_{0,1}^{1+p_3/2} w_{n,1}^{p_3/2} w_{n,2}^{p_3+1}
			 \rd x.
		\end{aligned}
	\]
On the other hand, if $\kappa_{n,1} = 0 < \kappa_{n,2}$, then 
$u_{0,1}^{p_3/2} u_{0,2}^{p_3+1} w_{n,1}^{1+p_3/2} 
\equiv 0 \equiv u_{0,1}^{1+p_3/2} w_{n,1}^{p_3/2} w_{n,2}^{p_3+1}$. 
Therefore, if $\kappa_{n,1} = 0 < \kappa_{n,2}$, then for some $\theta_5 \in (1,2)$, one has 
	\[
		\begin{aligned}
			o ( \kappa_{n,2}^{\theta_5}  )
			&= \sum_{i=1}^{2} \kappa_{n,i} \int_{\RN} 
			u_{0,1}^{p_3} u_{0,2}^{p_3+1} w_{n,1} + u_{0,1}^{p_3+1} u_{0,2}^{p_3} w_{n,2} 
			+ u_{0,2} w_{n,1}^{p_3+1} w_{n,2}^{p_3} \rd x
			\\
			& \quad + 
			 \int_{\RN} u_{0,1}^{p_3/2} u_{0,2}^{p_3+1} w_{n,1}^{1+p_3/2} 
			  + u_{0,1}^{1+\eta} w_{n,1}^{p_3/2} w_{n,2}^{p_3+1}
			 \rd x .		 
		\end{aligned}
	\]
Thus, in both cases, we infer that 
	\begin{equation}\label{eq:3.51}
		\begin{aligned}
			& - \frac{\beta \tau_{n,1}^{p_3+1} \tau_{n,2}^{p_3+1} }{p_3+1} 
			\int_{\RN} ( u_{0,1} + w_{n,1} )^{p_3+1} 
						 ( u_{0,2} + w_{n,2} )^{p_3+1} \rd x 
			\\
			\leq & \, 
			- \frac{\beta}{p_3+1} \int_{\RN} u_{0,1}^{p_3+1} u_{0,2}^{p_3+1} 
			+ w_{0,1}^{p_3+1} w_{0,2}^{p_3+1} \rd x 
			\\
			& \, 
			- \beta \int_{\RN} u_{0,1}^{p_3} u_{0,2}^{p_3+1} w_{n,1} 
			+ u_{0,1}^{p_3+1} u_{0,2}^{p_3} w_{n,2} 
			+ u_{0,2} w_{n,1}^{p_3+1} w_{n,2}^{p_3}
			\rd x 
			\\
			& \, 
			+ \beta \left( \frac{\kappa_{n,1}}{\alpha_1} + \frac{\kappa_{n,2}}{\alpha_2} \right) 
			\int_{\RN} u_{0,1}^{p_3+1} u_{0,2}^{p_3+1} + w_{0,1}^{p_3+1} w_{0,2}^{p_3+1} \rd x
			+ o ( \kappa_{n,1}^{\theta_5} + \kappa_{n,2}^{\theta_5} ).
		\end{aligned}
	\end{equation}

	From \eqref{eq:3.38}, \eqref{eq:3.43} and \eqref{eq:3.47}--\eqref{eq:3.51}, 
setting $\theta_0:= \min\{\theta_j \ |\ 1 \leq j \leq 5 \} \in (1,2)$, we have  
	\[
		\begin{aligned}
			& E( \tau_{n,1} [ u_{0,1} + w_{n,1} ], \ \tau_{n,2} [ u_{0,2} + w_{n,2} ] )
			\\
			\leq & \, 
			E(u_0) + E_\infty(w_0) 
			- \sum_{i=1}^2 \frac{\kappa_{n,i}}{\alpha_i} 
			\left( \lambda_{0,i} \alpha_i +  \| u_{0,i} \|_{V_i}^2 
			+ \| \nabla w_{0,i} \|_{L^2}^2 
			- \mu_i \int_{\RN} u_{0,i}^{2p_i+2} + w_{0,i}^{2p_i+2} \rd x
			\right)
			\\
			& \, - \int_{\RN} \mu_1 u_{0,1} w_{n,1}^{2p_1+1} 
			+ \mu_2 u_{0,2} w_{n,2}^{2p_2 + 1} 
			+ \beta u_{0,2} w_{n,1}^{p_3+1} w_{n,2}^{p_3} \rd x 
			\\
			& \, 
			+ \beta \left( \frac{\kappa_{n,1}}{\alpha_1} + \frac{\kappa_{n,2}}{\alpha_2} \right) 
			\int_{\RN} u_{0,1}^{p_3+1} u_{0,2}^{p_3+1} + w_{0,1}^{p_3+1} w_{0,2}^{p_3+1} \rd x
			+ \left\{\begin{aligned}
				& o(\kappa_{n,1}^{\theta_0}) & &\text{if $\kappa_{n,1} > 0$},
				\\
				& o( \kappa_{n,2}^{\theta_0} ) & &\text{if $\kappa_{n,1} = 0 < \kappa_{n,2}$}. 
			\end{aligned}\right.
		\end{aligned}		
	\]
By \eqref{eq:3.1} and \eqref{eq:3.3}, note that 
	\[
		\begin{aligned}
			\| u_{0,1} \|_{V_1}^2 &= - \lambda_{0,1} \| u_{0,1} \|_{L^2}^2 
			+ \int_{\RN} \mu_1 u_{0,1}^{2p_1+2} + \beta u_{0,1}^{p_3+1} u_{0,2}^{p_3+1} \rd x,
			\\
			\| u_{0,2} \|_{V_2}^2 &= - \lambda_{0,2} \| u_{0,2} \|_{L^2}^2
			+ \int_{\RN} \mu_2 u_{0,2}^{2p_2+2} + \beta u_{0,1}^{p_3+1} u_{0,2}^{p_3+1} \rd x,
			\\
			\| \nabla w_{0,1} \|_{L^2}^2 &= - \lambda_{0,1} \| w_{0,1} \|_{L^2}^2 
			+ \int_{\RN} \mu_1 w_{0,1}^{2p_1+2} + \beta w_{0,1}^{p_3+1} w_{0,2}^{p_3+1} \rd x,
			\\
			\| \nabla w_{0,2} \|_{L^2}^2 &= - \lambda_{0,2} \| w_{0,2} \|_{L^2}^2
			+ \int_{\RN} \mu_2 w_{0,2}^{2p_2+2} + \beta w_{0,1}^{p_3+1} w_{0,2}^{p_3+1} \rd x.
		\end{aligned}
	\]
Recalling $\alpha_i = \| u_{0,i} \|_{L^2}^2 + \| w_{0,i} \|_{L^2}^2$, 
we have 
	\[
		\| u_{0,i} \|_{V_i}^2 + \| \nabla w_{0,i} \|_{L^2}^2 
		=  - \lambda_{0,i} \alpha_i 
		+ \mu_i \int_{\RN} u_{0,i}^{2p_i+2} + w_{0,i}^{2p_i+2} \rd x 
		+ \beta \int_{\RN} u_{0,1}^{p_3+1} u_{0,2}^{p_3+1} 
		+ w_{0,1}^{p_3+1} w_{0,2}^{p_3+1} \rd x. 
	\]
Thus, 
	\begin{equation}\label{eq:3.52}
		\begin{aligned}
			& E( \tau_{n,1} [ u_{0,1} + w_{n,1} ], \ \tau_{n,2} [ u_{0,2} + w_{n,2} ] )
			\\
			\leq & \, 
			E(u_0) + E_\infty(w_0) - \int_{\RN} \mu_1 u_{0,1} w_{n,1}^{2p_1+1} 
			+ \mu_2 u_{0,2} w_{n,2}^{2p_2+1} + \beta u_{0,2} w_{n,1}^{p_3+1} w_{n,2}^{p_3} \rd x
			\\
			& \quad  
			+ \left\{\begin{aligned}
				& o(\kappa_{n,1}^{\theta_0}) & &\text{if $\kappa_{n,1} > 0$},
				\\
				& o( \kappa_{n,2}^{\theta_0} ) & &\text{if $\kappa_{n,1} = 0 < \kappa_{n,2}$}. 
			\end{aligned}\right.
		\end{aligned}
	\end{equation}

\medskip

\noindent
{\bf Step 4:} {\sl Conclusion}

\medskip

	Recalling \eqref{eq:3.36}, 
we first consider the case $\kappa_{n,1} > 0$. 
In this case, by Lemma \ref{lemma:3.1}, we notice that 
	\[
		c_0 (1+|x|)^{-(N-1)/2} \exp \left( - \sqrt{\lambda_{0,1}} |x| \right) 
		\leq u_{0,1}(x), \quad 
		w_{n,1}^{2p_1+1} (x) \leq \exp \left( - \sqrt{\lambda_{0,1} + \eta_0} |x| \right)
	\]
for some $\eta_0 > 0$. Hence, Lemma \ref{lemma:3.3} yields 
	\[
		\int_{\RN} u_{0,1} w_{n,1}^{2p_1 + 1} \rd x 
		\geq c (1+n)^{-(N-1)/2} \exp \left( - \sqrt{\lambda_{0,1} } n  \right). 
	\]
Since $\theta_0 > 1$,  it follows from Step 2 that 
	\[
		\frac{\kappa_{n,1}^{\theta_0}}{(1+n)^{-(N-1)/2}  \exp \left(  -\sqrt{\lambda_{0,1}} n \right)}
		\to 0 \quad {\rm as} \ n \to \infty.
	\]
Thus by \eqref{eq:3.52}, for sufficiently large $n$, we obtain 
	\[
		E( \tau_{n,1} [ u_{0,1} + w_{n,1} ], \ \tau_{n,2} [ u_{0,2} + w_{n,2} ] ) 
		< E(u_0) + E_\infty(w_0) = e( \gamma_{u_0} ) + e_\infty(\gamma_{w_0}).
	\]
Hence, \eqref{eq:3.35} holds.

	Next, suppose $\kappa_{n,1} = 0 < \kappa_{n,2}$. 
Since $w_{n,1} u_{0,1} \equiv 0$, \eqref{eq:3.52} becomes 
	\[
		E( u_{0,1} + w_{n,1}, \tau_{n,2} [ u_{0,2} + w_{n,2} ] ) 
		\leq E(u_0) + E_\infty(w_0) - \int_{\RN} \mu_2 u_{0,2} w_{n,2}^{2p_2+1} 
		+ \beta u_{0,2} w_{n,1}^{p_3+1} w_{n,2}^{p_3} \rd x 
		+ o( \kappa_{n,2}^{\theta_0} ).
	\]

	We first treat the case $u_{0,1} > 0$ and $w_{0,1} \equiv 0$, hence, 
we have 
	\begin{equation}\label{eq:3.53}
		E( u_{0,1} + w_{n,1}, \tau_{n,2} [ u_{0,2} + w_{n,2} ] ) 
		\leq E(u_0) + E_\infty(w_0) - \int_{\RN} \mu_2 u_{0,2} w_{n,2}^{2p_2+1} \rd x 
		+ o( \kappa_{n,2}^{\theta_0} ).
	\end{equation}

	When either $1 \leq p_3$ or $0<p_3<1$ and 
$\lambda_{0,3} = (1+p_3)^2 \lambda_{0,1} / (1-p_3)^2 > \lambda_{0,2}$, 
it follows from Lemmas \ref{lemma:3.1} and \ref{lemma:3.3} that 
for  some $c_1,c_2>0$, 
	\[
		\begin{aligned}
			\int_{\RN} u_{0,2} w_{n,2}^{2p_2+1} \rd x 
			&\geq \int_{\RN} c_1 w_{n,2}^{2p_2+1} (1+|x|)^{-(N-1)/2} \exp 
			\left( - \sqrt{\lambda_{0,2}} |x| \right) \rd x
			\\
			&\geq c_2 n^{-(N-1)/2} \exp \left( - \sqrt{\lambda_{0,2}} n \right).
		\end{aligned}
	\]
On the other hand, by Step 2 and $\theta_0>1$, we have 
	\[
		\frac{\kappa_{n,2}^{\theta_0}}{n^{-(N-1)/2} \exp \left( - \sqrt{\lambda_{0,2}} n \right)} 
		\to 0 \quad {\rm as} \ n \to \infty.
	\]
From \eqref{eq:3.53}, we observe that \eqref{eq:3.35} holds.

	Next suppose $0 < p_3 < 1$ and 
$\lambda_{0,3}  \leq \lambda_{0,2}$. 
In this case, Lemmas \ref{lemma:3.1} and \ref{lemma:3.3} imply that 
for every $ \lambda_{0,3} < \lambda_2$, 
	\[
		\int_{\RN} u_{0,2} w_{n,2}^{2p_2+1} \rd x 
			 \geq C_{\lambda_2} \exp \left( - \sqrt{\lambda_2} n \right).
	\]
Noting $\theta_0>1$ and Step 2, we see that 
	\[
		\kappa_{n,2}^{\theta_0} 
		\left( \int_{\RN} u_{0,2} w_{n,2}^{2p_2+1} \rd x  \right)^{-1} 
		\to 0 \quad {\rm as} \ n \to \infty,
	\]
which gives \eqref{eq:3.35} through \eqref{eq:3.53}.

	Finally, we treat the case $u_{0,1} \equiv 0$ and $0<u_{0,2},w_{0,1}, w_{0,2}$. 
In this case, \eqref{eq:3.52} becomes 
	\[
		E(w_{n,1} , \tau_{n,2} [u_{0,2} + w_{n,2}] ) 
		\leq E(u_0) + E_\infty(w_0) - \int_{\RN} \mu_2 u_{0,2} w_{n,2}^{2p_2+1} 
		+ \beta u_{0,2} w_{n,1}^{p_3+1} w_{n,2}^{p_3} \rd x 
		+ o (\kappa_{n,2}^{\theta_0}).
	\]

	When either $1 \leq p_3$ or else $0<p_3<1$ and 
$\lambda_{0,3} > \lambda_{0,2}$, 
we can use a similar argument and obtain 
	\[
		\kappa_{n,2}^{\theta_0} \left( \int_{\RN} u_{0,2} w_{n,2}^{2p_2+1}  \rd x \right)^{-1} 
		\to 0
	\]
and \eqref{eq:3.35} holds.

	On the other hand, when $0<p_3<1$ and 
$\lambda_{0,3} \leq \lambda_{0,2}$, 
by \eqref{eq:3.1} and \eqref{eq:3.3} with $u_{0,1} \equiv 0$, one has 
	\[
		\begin{aligned}
			\int_{\RN} \mu_2 u_{0,2} w_{n,2}^{2p_2+1} 
			+ \beta u_{0,2} w_{n,1}^{p_3+1} w_{n,2}^{p_3} \rd x
			&= \int_{\RN} \nabla w_{n,2} \cdot \nabla u_{0,2} 
			+ \lambda_{0,2} w_{n,2} u_{0,2} \rd x 
			\\
			&= \int_{\RN} - V_2(x) u_{0,2} w_{n,2} 
			+ \mu_2 u_{0,2}^{2p_2+1} w_{n,2} \rd x
			\\
			&\geq \mu_2 \int_{\RN} u_{0,2}^{2p_2+1} w_{n,2} \rd x.
		\end{aligned}
	\]
Hence, by Lemmas \ref{lemma:3.1} and \ref{lemma:3.3}, 
for every $\lambda_1 < \lambda_{0,3} < \lambda_2$, 
there exist $C_{\lambda_1}, C_{\lambda_2} > 0$ such that 
	\[
		C_{\lambda_2} \exp \left( - \sqrt{\lambda_2} n \right) 
		\leq \int_{\RN} u_{0,2}^{2p_2+1} w_{n,2} \rd x 
		\leq C_{\lambda_1} \exp \left( - \sqrt{\lambda_1} n \right). 
	\]
Thus, Step 2 yields 
	\[
		\kappa_{n,2}^{\theta_0} 
		\left( \int_{\RN} u_{0,2}^{2p_2+1} w_{n,2} \rd x \right)^{-1} \to 0 
		\quad {\rm as} \ n \to \infty
	\]
and \eqref{eq:3.35} holds.

Since \eqref{eq:3.35} gives a contradiction, 
the dichotomy does not happen and compactness occurs. 
\end{proof}

\section{Proof of Theorem \ref{theorem:1.2}}
\label{section:4}

	In this section, we shall prove Theorem \ref{theorem:1.2} and hereafter we always 
suppose (p1) and (V2). 
Let $(u_n) \subset \wt{M}(\alpha)$ be any minimizing sequence for $\tea$. 
From Lemmas \ref{lemma:2.1}, \ref{lemma:2.2} and \ref{lemma:2.3}, we may suppose that 
$(u_n)$ is bounded in $\tH$ and satisfies 
	\begin{equation}\label{eq:4.1}
		\| E'(u_n) - \lambda_{n,1} Q_1'(u_n) - \lambda_{n,2} Q_2'(u_n) \|_{(\tH)^\ast} \to 0, \quad 
		\| (u_{n,i})_- \|_{L^2} \to 0.
	\end{equation}
Furthermore, let $u_{n} \rightharpoonup u_0$ weakly in $\tH$ and 
$\lambda_{n,i} \to \lambda_{0,i}$. 
By $V_2(x) \to \infty$ as $|x| \to \infty$ in (V2) and the result in \cite{BW-95}, one sees 
	\begin{equation}\label{eq:4.2}
		\left\| u_{n,2} - u_{0,2} \right\|_{L^q} \to 0 \quad 
		\text{for all } q \in [2, 2^\ast ). 
	\end{equation}
Hence, $\alpha_2 = \| u_{0,2} \|_{L^2}^2$. 
Thus, what remains to prove is $\alpha_1 = \| u_{0,1} \|_{L^2}^2$ 
since this leads to $\| u_{n,1} - u_{0,1} \|_{L^2} \to 0$ and 
$\| u_n - u_0 \|_{\tH} \to 0$ from $E(u_n) \to \tea$.

	We argue indirectly and suppose 
	\begin{equation}\label{eq:4.3}
		\| u_{0,1}  \|_{L^2}^2 < \alpha_1.
	\end{equation}

We first show 
	\begin{lemma}\label{lemma:4.1}
		$\lim_{n\to\infty} \sup_{z \in \Z^N} \| u_{n,1} - u_{0,1} \|_{L^2(z+Q)} > 0$. 
	\end{lemma}

	\begin{proof}
If $\sup_{z \in \Z^N} \| u_{n,1} - u_{0,1}  \|_{L^2(z+Q)} \to 0$, then we have 
$ \| u_{n,1} - u_{0,1} \|_{L^q} \to 0$ for any $2<q < 2^\ast$. 
Therefore, the weak lower semicontinuity of the norms and \eqref{eq:4.2} yield 
	\[
		\wt{e} \left(  \| u_{0,1} \|_{L^2}^2 , \alpha_2 \right) 
		\leq E( u_0 ) 
		\leq \liminf_{n \to \infty} E(u_n) = \tea. 
	\]
Next, remark that $e_1(\gamma) < 0$ holds for any $\gamma > 0$. 
See \cite{C-03,S14}. 
Thus, Lemma \ref{lemma:2.1} (iii) and \eqref{eq:4.3} give a contradiction:  
	\[
		\wt{e} \left( \| u_{0,1} \|_{L^2}^2 , \alpha_2 \right) 
		\leq \tea \leq 
		 \wt{e} \left(  \| u_{0,1} \|_{L^2}^2 , \alpha_2 \right)  + 
		e_1 \left( \alpha_1 - \| u_{0,1} \|_{L^2}^2 \right)
		< \wt{e} \left( \| u_{0,1} \|_{L^2}^2 , \alpha_2 \right) .
	\]
Hence, Lemma \ref{lemma:4.1} holds. 
	\end{proof}

	From Lemma \ref{lemma:4.1}, 
we may find $(y_n)_{n=1}^\infty \subset \RN$ such that 
	\[
		|y_n| \to \infty, \quad u_{n,1}( \cdot +y_n) \rightharpoonup w_{0,1} \not\equiv 0
		\quad \text{weakly in } H^1(\RN).
	\]
Next, as in Lemma \ref{lemma:3.6}, we shall prove

	\begin{lemma}\label{lemma:4.2}
		We have 
		$\| u_{n,1} - u_{0,1} - w_{0,1}(\cdot - y_n) \|_{H^1} \to 0$ and 
		\begin{equation}\label{eq:4.4}
			\tea = \wt{e} \left( \| u_{0,1} \|_{L^2}^2 , \alpha_2 \right) 
			+ e_1 \left( \| w_{0,1} \|_{L^2}^2 \right) 
			= E(u_0) + E_1(w_{0,1}).
		\end{equation}
	\end{lemma}

	\begin{proof}
Since a proof is similar to that of Lemma \ref{lemma:3.6}, we only give a sketch of the proof. 
We first suppose that 
$ \lim_{n\to \infty} \| u_{n,1} - u_{0,1} - w_{0,1} (\cdot - y_n) \|_{L^2} = \delta_0 >0$. 
By $V_1(x) \to 0$ as $|x| \to \infty$ due to (V2), we may prove 
	\[
		\begin{aligned}
			&E(u_n) - E(u_0) - E( w_{0,1} (\cdot - y_n) , 0 ) - E( u_{n,1} - u_{0,1} - w_{0,1} (\cdot - y_n) ) , 0) \to 0, 
			\\
			&0 < \delta_0 = \alpha_1 - \| u_{0,1} \|_{L^2}^2 - \| w_{0,1} \|_{L^2}^2.
		\end{aligned}
	\]
From $|y_n| \to \infty$ and $u_{n,1} \to u_{0,1}$ in $L^q_{\rm loc}(\RN)$ for $1 \leq q < 2^\ast$, 
it follows that 
	\[
		\begin{aligned}
			&\liminf_{n \to \infty} E( w_{0,1} (\cdot - y_n) , 0 ) 
			= \liminf_{n \to \infty} E_1( w_{0,1} (\cdot - y_n)  ) \geq e_1 \left(\| w_{0,1} \|_{L^2}^2\right), 
			\\
			&\liminf_{n\to \infty} E( u_{n,1} - u_{0,1} - w_{0,1} (\cdot - y_n) ) , 0) 
			= \liminf_{n\to\infty} E_1( u_{n,1} - u_{0,1} - w_{0,1} (\cdot - y_n) ) \geq e_1(\delta_0).
		\end{aligned}
	\]
Hence, we obtain 
	\[
		\tea \geq \wt{e} \left( \| u_{0,1} \|_{L^2}^2 , \alpha_2 \right) + e_1 \left(\| w_{0,1} \|_{L^2}^2\right) 
		+ e_1 ( \delta_0 ).
	\]

	For $e_1(\alpha)$, it is known that $e_1(\gamma + \delta) < e_1(\gamma ) + e_1(\delta)$ 
(see \cite{C-03,S14}). Combining this with Lemma \ref{lemma:2.1} (iii), we obtain a contradiction:
	\[
		\wt{e} \left( \| u_{0,1} \|_{L^2}^2 , \alpha_2 \right) 
		+ e_1 \left(  \|w_{0,1} \|_{L^2}^2 + \delta_0 \right) 
		< \tea \leq \wt{e} \left( \| u_{0,1} \|_{L^2}^2 , \alpha_2 \right) 
		+ e_1 \left(  \|w_{0,1} \|_{L^2}^2 + \delta_0 \right).
	\]
Therefore, $\| u_{n,1} - u_{0,1} - w_{0,1} (\cdot - y_n) \|_{L^2} \to 0$.

	The rest of the proof is identical to that of Lemma \ref{lemma:3.6} 
and we omit the detail. 
	\end{proof}

	To proceed, by \eqref{eq:4.1} and $\| u_{n,2} - u_{0,2} \|_{L^2} \to 0$, 
we observe that 
$u_0$ satisfies \eqref{eq:3.1} and $w_{0,1}$ is a positive solution of 
	\begin{equation}\label{eq:4.5}
		-\Delta w_{0,1} + \lambda_{0,1} w_{0,1} = \mu_1 w_{0,1}^{2p_2+1} 
		\quad \text{in } \RN. 
	\end{equation}
Next we show $\lambda_{0,1}>0$.

	\begin{lemma}\label{lemma:4.3}
		We have $\lambda_{0,1} > 0$. 
	\end{lemma}

	\begin{proof}
Since $w_{0,1} \in H^1(\RN)$ is a positive solution of \eqref{eq:4.5} with $2p_1 + 1 < 2^\ast - 1$, 
as in \cite{BL-83}, when $N \geq 2$, we may show that $w_{0,1}$ satisfies the Pohozaev identity 
	\begin{equation}\label{eq:4.6}
		0 = \frac{N-2}{2N} \| \nabla w_{0,1} \|_{L^2}^2 
		-  \frac{\mu_1}{2p_2+2} \|w_{0,1} \|_{L^{2p_1+2}}^{2p_1+2} 
		+ \frac{\lambda_{0,1}}{2} \| w_{0,1} \|_{L^2}^2. 
	\end{equation}
We note that when $N=1$, \eqref{eq:4.6} also holds since $w_{0,1}$ satisfies 
	\[
		\frac{1}{2} \left( w_{0,1}'(x) \right)^2 + \frac{\mu_1}{2p_1+2} w_{0,1}^{2p_1+2} (x) 
		- \frac{\lambda_{0,1}}{2} w_{0,1}^2(x) \equiv 0 \quad \text{in} \ \R.
	\]
Thus, from $\| \nabla w_{0,1} \|_{L^2}^2 + \lambda_{0,1} \| w_{0,1} \|_{L^2}^2 
	= \mu_1 \| w_{0,1} \|_{L^{2p_1+2}}^{2p_1+2}$, it follows that 
	\[
		0= \left( \frac{N-2}{2N} - \frac{1}{2p_1 +2} \right) \| \nabla w_{0,1} \|_{L^2}^2 
		+ \left( \frac{1}{2} - \frac{1}{2p_1 + 2} \right) \lambda_{0,1} \| w_{0,1} \|_{L^2}^2,
	\]
which implies $\lambda_{0,1} > 0$ due to $w_{0,1} \not\equiv 0$. 
	\end{proof}

	By (V2), especially $V_2(x) \to \infty$ as $|x| \to \infty$, 
$\lambda_{0,1} > 0$ and \eqref{eq:3.1}, as in the proof of Lemma \ref{lemma:3.1} (Steps 1 and 2) and \cite{GNN-81}, 
if $u_{0,1} \not \equiv 0$, then for each $\lambda \in (0,\lambda_{0,1})$ 
we may find a $C_\lambda > 0$ such that for all $x \in \RN$, 
	\begin{equation}\label{eq:4.7}
		\begin{aligned}
			&u_{0,1}(x) + u_{0,2}(x) \leq C_\lambda \exp \left( - \sqrt{\lambda} |x| \right) ,
			\\
			& c_0 \left( 1 + |x| \right)^{-(N-1)/2} \exp \left( - \sqrt{\lambda_{0,1}} |x| \right) 
			\leq u_{0,1} (x) ,
			\\
			&c_0 \left( 1 + |x| \right)^{-(N-1)/2} \exp \left( - \sqrt{\lambda_{0,1}} |x| \right)  
			\leq w_{0,1} (x) \leq 
			c_1 \left( 1 + |x| \right)^{-(N-1)/2} \exp \left( - \sqrt{\lambda_{0,1}} |x| \right)
		\end{aligned}
	\end{equation}
where $0<c_0 \leq c_1$. 
Now we derive a contradiction and complete the proof of Theorem \ref{theorem:1.2}:

\begin{proof}[Proof of Theorem \ref{theorem:1.2}]
As in the proof of Theorem \ref{theorem:1.1}, we set 
	\[
		w_n(x) := w_{0,1} (x - n \mathbf{e}_1 ), \quad 
		\tau_n := \frac{\sqrt{\alpha_1}}{\| u_{0,1} + w_n \|_{L^2}}, \quad 
		\kappa_n := \la u_{0,1} , w_n \ra_{L^2}.
	\]
Then 
	\[
		\kappa_n \to 0, \quad 
		\tau_n = 1 - \frac{\kappa_n}{\alpha_1} + O( \kappa_n^2 ).
	\]
Since $( \tau_n ( u_{0,1} + w_n )  , u_{0,2}  ) \in \wt{M}(\alpha)$, it follows that 
	\begin{equation}\label{eq:4.8}
		\begin{aligned}
			\tea & \leq E( \tau_n ( u_{0,1} + w_n )  , u_{0,2}  ) 
			\\
			&= \frac{\tau_n^2}{2} \|  u_{0,1} + w_n \|_{V_1}^2 + \frac{1}{2} \| u_{0,2} \|_{V_2}^2 
			- \frac{\mu_1 \tau_n^{2p_1+2} }{2p_1+2} \int_{\RN} \left( u_{0,1} + w_n \right)^{2p_1+2} \rd x 
			- \frac{ \mu_2 }{2p_2+2} \int_{\RN} u_{0,2}^{2p_2+2} \rd x 
			\\
			&\quad - \frac{\beta \tau_n^{p_3+1} }{p_3+1} 
			\int_{\RN}  \left(u_{0,1} + w_n \right)^{p_3+1}  u_{0,2}^{p_3+1} \rd x
			\\
			&= \frac{1}{2}  \left( 1 - \frac{2\kappa_n}{\alpha_1} \right)  
			\left( \| u_{0,1} \|_{V_1}^2 + 2 \la u_{0,1} , w_n \ra_{V_1} + \| w_n \|_{V_1}^2 \right)
			+ \frac{1}{2} \| u_{0,2} \|_{V_2}^2  
				- \frac{\mu_2}{2p_2+2} \int_{\RN} u_{0,2}^{2p_2+2} \rd x
			\\
			&\quad 
			- \frac{\mu_1}{2p_1+2} \left( 1 - \frac{2p_1+2}{\alpha_1} \kappa_n \right) 
			\int_{\RN} \left( u_{0,1} + w_n \right)^{2p_1+2} \rd x 
			\\
			&\quad 
			- \frac{\beta}{p_3+1} \left( 1 - \frac{p_3+1}{\alpha_1} \kappa_n \right) 
			\int_{\RN} \left( u_{0,1} + w_n \right)^{p_3+1}  u_{0,2}^{p_3+1} \rd x + O(\kappa_n^2).
		\end{aligned}
	\end{equation}

	When $u_{0,1} \equiv 0$, we have $\kappa_{n} = 0$. 
Since $u_{0,2} > 0$ and $w_n>0$, \eqref{eq:4.8} becomes 
	\[
		\begin{aligned}
			&\tea 
			\\
			\leq \, &
			\frac{1}{2} \| w_n \|_{V_1}^2  + \frac{1}{2} \| u_{0,2} \|_{V_2}^2 
			- \frac{\mu_2}{2p_2+2} \int_{\RN} u_{0,2}^{2p_2+2} \rd x 
			- \frac{\mu_1}{2p_1+2} \int_{\RN} w_n^{2p_1+2} \rd x 
			- \frac{\beta}{p_3+1} \int_{\RN} w_n^{p_3+1} u_{0,2}^{p_3+1} \rd x
			\\
			< \, & E(u_0) + E_1(w_{0,1}).
		\end{aligned}
	\]
Hence, this contradicts \eqref{eq:4.4}.

	Next, we suppose $u_{0,1} \not\equiv 0$ and 
it follows from $u_{0,1},w_n > 0$, \eqref{eq:4.7} and \eqref{eq:4.5} that 
	\begin{equation}\label{eq:4.9}
		\begin{aligned}
			\la u_{0,1} , w_n \ra_{V_1} 
			&= - \lambda_{0,1} \kappa_n + \int_{\RN} \mu_1 u_{0,1}^{2p_1+1} w_{n} 
			+ \beta u_{0,1}^{p_3} u_{0,2}^{p_3+1} w_n \rd x 
			= O( \kappa_n^{1/2} ),
			\\
			\| u_{0,1} \|_{V_1}^2 
			&= - \lambda_{0,1} \| u_{0,1} \|_{L^2}^2 
			+ \int_{\RN} \mu_1 u_{0,1}^{2p_1+2} + \beta u_{0,1}^{p_3+1} u_{0,2}^{p_3+1} \rd x,
			\\
			\| \nabla w_{n} \|_{L^2}^2 
			&= - \lambda_{0,1} \| w_{n} \|_{L^2}^2 + \int_{\RN} \mu_1 w_{n}^{2p_1+2} \rd x.
		\end{aligned}
	\end{equation}
Therefore, by $V_1(x) \leq 0$, we obtain 
	\begin{equation}\label{eq:4.10}
		\begin{aligned}
			& \frac{1}{2} \left( 1 - \frac{2\kappa_n}{\alpha_1} \right) 
			\left( \| u_{0,1} \|_{V_1}^2 + 2 \la u_{0,1} , w_n \ra_{V_1} + \| w_n \|_{V_1}^2 \right) 
			\\
			\leq \, & \frac{1}{2} \left( \| u_{0,1} \|_{V_1}^2  +  \| \nabla w_{0,1} \|_{L^2}^2 \right) 
			+ \la u_{0,1} , w_n \ra_{V_1} - \frac{\kappa_n}{\alpha_1} 
			\left( \| u_{0,1} \|_{V_1}^2 + \| \nabla w_{0,1}\|_{L^2}^2 \right) + O(\kappa_n^{3/2})
			\\
			= \, & \frac{1}{2} \left( \| u_{0,1} \|_{V_1}^2  +  \| \nabla w_{0,1} \|_{L^2}^2 \right) 
			 + \mu_1 \int_{\RN} u_{0,1}^{2p_1+1} w_n \rd x + \beta \int_{\RN} u_{0,1}^{p_3} u_{0,2}^{p_3+1} w_n \rd x
			 \\
			 & \quad 
			- \frac{\kappa_n}{\alpha_1} 
			\left( \lambda_{0,1} \alpha_1 + \| u_{0,1} \|_{V_1}^2 + \| \nabla w_{0,1}\|_{L^2}^2 \right) 
			 + O(\kappa_n^{3/2}).
		\end{aligned}
	\end{equation}
By Lemmas \ref{lemma:3.3} and \ref{lemma:3.4}, \eqref{eq:4.7} and $\kappa_{n} \to 0$, we observe that 
	\begin{equation}\label{eq:4.11}
		\begin{aligned}
			&- \frac{\mu_1}{2p_1+2} \left( 1 - \frac{2p_1+2}{\alpha_1} \kappa_n \right) 
			\int_{\RN} \left( u_{0,1} + w_n \right)^{2p_1+2} \rd x 
			\\
			 \leq \,& 
			 -\frac{\mu_1}{2p_1+2} \int_{\RN} u_{0,1}^{2p_1+2} + w_n^{2p_1+2} + (2p_1+2) 
			 \left( u_{0,1}^{2p_1+1} w_n + w_n^{2p_1+1} u_{0,1} \right) \rd x 
			 \\
			 & \quad + \frac{\kappa_n}{\alpha_1} \mu_1 \int_{\RN} u_{0,1}^{2p_1+2} + w_n^{2p_1+2} \rd x 
			 	+ O(\kappa_n^{\theta}) 
		\end{aligned}
	\end{equation}
for some $\theta > 1$. Next, by 
	\[
		a^{p_3+1} + (p_3+1) a^{p_3} b \leq (a+b)^{p_3+1} 
		\leq a^{p_3+1} + (p_3+1) (a+b)^{p_3} b \quad 
		\text{for each $a,b \geq 0$}, 
	\]
we obtain 
	\begin{equation}\label{eq:4.12}
		\begin{aligned}
			& - \frac{\beta}{p_3+1} \int_{\RN} \left( u_{0,1} + w_n \right)^{p_3+1}  u_{0,2}^{p_3+1} \rd x 
			\\
			\leq \, & -\frac{\beta}{p_3+1} \int_{\RN} 
			u_{0,1}^{p_3+1} u_{0,2}^{p_3+1} +(p_3+1) u_{0,1}^{p_3} u_{0,2}^{p_3+1} w_n \rd x 
		\end{aligned}
	\end{equation}
and 
	\begin{equation}\label{eq:4.13}
		\frac{\kappa_n}{\alpha_1} \beta \int_{\RN} (u_{0,1}+w_n)^{p_3+1} u_{0,2}^{p_3+1} \rd x 
		\leq \frac{\kappa_n}{\alpha_1} \beta 
		\int_{\RN} u_{0,1}^{p_3+1} u_{0,2}^{p_3+1} \rd x + O(\kappa_n^{\theta}).
	\end{equation}

	From \eqref{eq:4.8}--\eqref{eq:4.13} and $\alpha_1 = \| u_{0,1} \|_{L^2}^2 + \| w_{0,1} \|_{L^2}^2$, 
it follows that 
	\[
		\begin{aligned}
			\tea 
			&\leq 
			\frac{1}{2} \left( \| u_{0,1} \|_{V_1}^2 + \| \nabla w_{0,1} \|_{L^2}^2 \right) 
			+ \mu_1 \int_{\RN} u_{0,1}^{2p_1+1} w_n \rd x + \beta \int_{\RN} u_{0,1}^{p_3} u_{0,2}^{p_3+1} w_n \rd x 
			\\
			& \quad  
			- \frac{\kappa_n}{\alpha_1} 
			\left( \alpha_1 \lambda_{0,1} + \| u_{0,1} \|_{V_1}^2 + \| \nabla w_{0,1} \|_{L^2}^2 \right)
			+ \frac{1}{2} \| u_{0,2} \|_{V_2}^2 - \frac{\mu_2}{2p_2+2} \int_{\RN} u_{0,2}^{2p_2+2} \rd x 
			\\
			&\quad 
			- \frac{\mu_1}{2p_1+2} \int_{\RN} u_{0,1}^{2p_1+2} + w_{n}^{2p_1+2} 
			+ (2p_1+2) \left( u_{0,1}^{2p_1+1} w_n + w_n^{2p_1+1} u_{0,1} \right) \rd x 
			\\
			&\quad 
			+ \frac{\kappa_n}{\alpha_1} \mu_1 \int_{\RN} u_{0,1}^{2p_1+2} + w_n^{2p_1+2} \rd x 
			\\
			&\quad  
			- \frac{\beta}{p_3+1} 
			\int_{\RN} u_{0,1}^{p_3+1} u_{0,2}^{p_3+1} + (p_3+1) u_{0,1}^{p_3} u_{0,2}^{p_3+1} w_n \rd x 
			+ \frac{\kappa_n}{\alpha_1} \beta \int_{\RN} u_{0,1}^{p_3+1} u_{0,2}^{p_3+1} \rd x 
			+ O(\kappa_n^{\theta})
			\\
			&= E(u_0) + E_1(w_{0,1}) - \mu_1 \int_{\RN} w_n^{2p_1+1} u_{0,1} \rd x 
			 + O(\kappa_n^{\theta}).
		\end{aligned}
	\]
As in the proof of Theorem \ref{theorem:1.1}, since $\theta>1$, we see 
	\[
		- \mu_1 \int_{\RN} w_n^{2p_1+1} u_{0,1} \rd x + O(\kappa_n^{\theta}) < 0
	\]
for sufficiently large $n$, hence, \eqref{eq:4.4} gives a contradiction:
	\[
		\tea < E(u_0) + E_1(w_{0,1}) = \tea.
	\]
Thus, \eqref{eq:4.3} may not happen and we get $\| u_{n,1} - u_{0,1} \|_{L^2} \to 0$. 
Hence, Theorem \ref{theorem:1.2} holds. 
\end{proof}

\subsection*{Acknowledgement}
The first author was supported by JSPS KAKENHI Grant Number JP16K17623 and 
JP17H02851. 
The second author was supported by JSPS KAKENHI Grant Numbers 16K05225 and 17KK0086.

\appendix

\section{Proof of \eqref{eq:1.4} under $p_3 \geq 1$}
\label{section:A}

	In this appendix, following \cite{L-84-2} (cf. \cite{B-17,C-03}), 
we shall prove \eqref{eq:1.4} under (p1), (V1) and $p_3 \geq 1$. 
Here we allow the case $V_1 (x) \equiv V_2(x) \equiv 0$ and in this case, 
we may assume $\theta \neq (0,0), (1,1)$ for \eqref{eq:1.4}.

	We begin with 
	\begin{lemma}\label{lemma:A.1}
		Suppose \textup{(V1)} and \textup{(p1)}. 
			\begin{enumerate}
			\item 
			Let $\gamma = (\gamma_1,\gamma_2) \in [0,\alpha_1] \times [0,\alpha_2] \setminus \{(0,0)\}$. 
			Then $e(\gamma) \leq e_\infty (\gamma) < 0$. 
			\item 
			Let $\gamma \in [0,\alpha_1] \times [0,\alpha_2] \setminus \{(0,0)\}$ with $\gamma_{i_0} > 0$. 
			Then there exists a $\delta_\gamma > 0$ such that 
			$ \| u_{i_0} \|_{L^{2p_{i_0}+2}}^{2p_{i_0}+2} \geq \delta_\gamma$ 
			for all $u \in M(\gamma)$ with $E(u) \leq e(\gamma) + \delta_\gamma$ 
			(resp. $E_\infty(u) \leq e_\infty(\gamma) + \delta_\gamma$). 
			\item 
			The values $e(\alpha_1,0), e(0,\alpha_2), e_\infty(\alpha_1,0), e_\infty(0,\alpha_2)$ are 
			attained by positive functions. 
		\end{enumerate}
	\end{lemma}

	\begin{proof}
(i) See Lemma \ref{lemma:2.1} and \cite{GJ-16}. 

(ii) We only prove for $E$ and the case $\gamma_1>0$, and argue by contradiction. 
Let $\gamma_1>0$. If there exists $(u_n)_{n=1}^\infty \subset M(\gamma) $ such that 
$E(u_n) \to e(\gamma)$ and $\| u_{n,1} \|_{L^{2p_1+2}} \to 0$, then we obtain 
	\[
		e(\gamma) = \lim_{n\to \infty} E(u_n) 
		\geq \liminf_{n\to \infty} E(0,u_{n,2}) \geq e( 0, \gamma_2 ).
	\]
On the other hand, Lemma \ref{lemma:2.1} and (i) give a contradiction:
	\[
		e(\gamma) \leq e(0,\gamma_2) + e_\infty(\gamma_1,0) < e(0,\gamma_2) \leq e(\gamma).
	\]
Thus, (ii) holds.

	(iii) First, using the argument in \cite[p.228--p.229]{L-84-2} and \cite[Chapter 8, section 3]{C-03}, 
we observe that the values $e(\alpha_1,0), e(0,\alpha_2), e_\infty(\alpha_1,0), e_\infty(0,\alpha_2)$ 
admit minimizers. Next, if $u$ is a minimizer of one of these values, then 
so is $|u|$. In addition, $|u|$ satisfies the elliptic equations and 
the strong maximum principle asserts $|u|>0$ in $\RN$. 
Hence, $|u|$ is the desired minimizer. 
	\end{proof}

For $i=1,2$, we set 
	\[
		\begin{aligned}
			\wt{E}_i(u) &:= \frac{1}{2} \|u_i \|_{V_i}^2 - \frac{\mu_i}{2p_i+2} \| u_i \|_{L^{2p_i+2}}^{2p_i+2} 
			- \frac{\beta}{p_3+1} \| u_1 u_2 \|_{L^{p_3+1}}^{p_3+1},
			\\
			\wt{E}_{\infty,i}(u) &:= \frac{1}{2} \| \nabla u_i \|_{L^2}^2 - \frac{\mu_i}{2p_i+2} \| u_i \|_{L^{2p_i+2}}^{2p_i+2} 
			- \frac{\beta}{p_3+1} \| u_1 u_2 \|_{L^{p_3+1}}^{p_3+1}.
		\end{aligned}
	\]
Then $\wt{E}_i (u) \leq \wt{E}_{\infty,i} (u)$, $E(u) = \wt{E}_1(u) + E(0,u_2) = \wt{E}_2(u) + E(u_1,0)$ and the same equalities hold for $E_\infty$ and 
$\wt{E}_{\infty,i}$.

Now fix $\theta = (\theta_1,\theta_2) \in [0,1] \times [0,1]$ with $\theta \neq (0,0)$ and assume 
$0< \theta_1 < 1$. Put $\gamma := ( \theta_1 \alpha_1 , \theta_2 \alpha_2)$ and 
choose $u=(u_1,u_2), v = (v_1,v_2) \in ( C^\infty_0(\RN) )^2$ so that 
	\begin{equation}\label{eq:A.1}
		\| u_i \|_{L^2}^2 = \theta_i \alpha_i, 
		\  \| v_i \|_{L^2}^2 = (1-\theta_i) \alpha_i , 
		\  E(u) \leq e( \gamma ) + \delta_{\gamma},
		\  E_\infty(v) \leq e_\infty ( \alpha - \gamma ) 
		+ \delta_{ \alpha - \gamma } 
	\end{equation}
where $\delta_{ \gamma }, \delta_{ \alpha - \gamma } > 0$ are constants in Lemma \ref{lemma:A.1}(ii). 
Since $u$ and $v$ have compact support, we find an $n$ so that 
$u_i (x) v_j(x-n\mathbf{e}_1) \equiv 0$ for $i,j=1,2$ and write $v_{n,i} (x) := v_i(x-n\mathbf{e}_1)$. 
From Lemma \ref{lemma:A.1} (ii), $p_3 \geq 1$ and $0<\theta_1<1$, it follows that 
	\[
		\begin{aligned}
			\theta_1 e( \alpha ) 
			& \leq \theta_1 E \left( \frac{u_1}{\sqrt{\theta_1}} , u_2 + v_{n,2} \right) 
			\\
			&= \theta_1 \left( \wt{E}_1 \left( \frac{u_1}{\sqrt{\theta_1}} , u_2 + v_{n,2} \right) + E(0,u_2 + v_{n,2}) \right)
			\\
			&\leq \wt{E}_1(u) - \mu_1 \frac{\theta_1^{-p_1} - 1 }{2p_1+2} \| u_1 \|_{L^{2p_1+2}}^{2p_1+2} 
			- \beta \frac{\theta_1^{-\frac{p_3-1}{2}}  - 1  }{p_3+1}  \| u_1 u_2 \|_{L^{p_3+1}}^{p_3+1} 
			+ \theta_1 \left( E(0,u_2) + E_\infty(0,v_{n,2}) \right)
			\\
			&\leq \wt{E}_1(u) - \mu_1 \frac{ \theta_1^{-p_1} - 1  }{2p_1+2} \delta_\gamma 
			+ \theta_1 \left( E(0,u_2) + E_\infty(0,v_{n,2}) \right). 
		\end{aligned}
	\]
Similarly, 
	\[
		\begin{aligned}
			(1-\theta_1) e(\alpha) 
			&\leq (1-\theta_1) E \left( \frac{v_{n,1}}{\sqrt{1-\theta_1}} , u_2 + v_{n,2} \right)
			\\
			&\leq \wt{E}_{\infty,1}(v) - \mu_1 \frac{ (1-\theta_1)^{-p_1} - 1  }{2p_1+2} \delta_{\alpha-\gamma} 
			+ \left( 1 - \theta_1 \right) \left( E(0,u_2) + E_\infty (0,v_{n,2}) \right).
		\end{aligned}
	\]
Hence, 
	\[
		\begin{aligned}
			e(\alpha) &= \theta_1 e(\alpha) + (1-\theta_1) e(\alpha)
			\\
			&\leq E(u) + E_\infty(v) 
			- \mu_1 \frac{  \theta_1^{-p_1} - 1  }{2p_1+2} \delta_\gamma 
			- \mu_1 \frac{ (1-\theta_1)^{-p_1} - 1  }{2p_1+2} \delta_{\alpha-\gamma} .
		\end{aligned}
	\]
Since $u,v \in (C^\infty_0(\RN))^2$ are arbitrary elements satisfying \eqref{eq:A.1}, 
from $0< \theta_1 < 1$ we infer that 
	\[
		\begin{aligned}
			e(\alpha) 
			&\leq  
			e(\gamma) + e_\infty(\alpha - \gamma) 
			- \mu_1 \frac{  \theta_1^{-p_1} - 1   }{2p_1+2} \delta_\gamma 
			- \mu_1 \frac{ (1-\theta_1)^{-p_1} - 1  }{2p_1+2} \delta_{\alpha-\gamma}
			\\
			&< e(\gamma) + e_\infty (\alpha - \gamma) = 
			e \left( \theta_1 \alpha_1 , \theta_2 \alpha_2 \right) 
			+ e_\infty \left( (1-\theta_1) \alpha_1 , (1-\theta_2) \alpha_2 \right).
		\end{aligned}
	\]
Hence, \eqref{eq:1.4} holds when $0<\theta_1 < 1$.

	Remark that in a similar way, we may prove \eqref{eq:1.4} when $0< \theta_2 < 1$. 
Thus, the remaining cases are $\theta = (1,0), (0,1), (1,1)$.

	When $\theta = (1,0), (0,1)$, from Lemma \ref{lemma:A.1} (iii), to prove \eqref{eq:1.4}, 
we can argue as in Step 1 in the proof of Theorem \ref{theorem:1.1}. 
Therefore, in the case $V_1 (x) \equiv V_2 (x) \equiv 0$, \eqref{eq:1.4} holds and $e_\infty (\alpha)$ is attained.

	Finally, let us consider the case $V_1 (x) \not\equiv 0$ or $V_2(x) \not\equiv 0$. 
Notice that $e_\infty (\alpha)$ is attained and using a similar argument to 
Lemma \ref{lemma:A.1} (iii), 
we may choose a positive minimizer for $e_\infty(\alpha)$. 
Again, we may exploit the proof of Step 1 in the proof of Theorem \ref{theorem:1.1} and 
conclude that \eqref{eq:1.4} holds.

\end{document}